\documentclass[12pt,a4paper]{article}
\usepackage[english]{babel}
\usepackage{amsthm}
\usepackage{amsmath}
\usepackage{amsfonts}
\usepackage{amssymb}
\usepackage{makeidx}
\usepackage{graphicx}
\usepackage[normalem]{ulem}
\usepackage{color}
\usepackage[colorlinks=true,linkcolor=blue,citecolor=blue,urlcolor=black]{hyperref}
\usepackage[left=3.2cm,right=3.4cm,top=3cm,bottom=2.5cm]{geometry}
\usepackage{fancyhdr}
\providecommand{\ams}[1]{\textit{AMS 2010 Subject Classification.} #1}
\providecommand{\keywords}[1]{\textit{Keywords.} #1}

\def\sec{\setcounter{equation}{0}\setcounter{figure}{0}}
\def\mP{\mathbb P}
\newcommand{\EXCLUDE}[1]{}
\newcommand{\no}{\nonumber}

\newcommand{\pr}[1]{\mathbb{P}\left\{ #1 \right\}}

\newcommand{\EXP}[1]{\mathbb{E}\!\left\{#1\right\} }

\newcommand{\VAR}[1]{\mathsf{VAR}\!\left(#1\right) }
\newcommand{\remove}[1]{}
\def\P{\mathbb{P}}
\def\E{\mathbb{E}}
\newcommand{\beq}{\begin{eqnarray}}
\newcommand{\eeq}{\end{eqnarray}}
\newcommand{\beqq}{\begin{eqnarray*}}
	\newcommand{\eeqq}{\end{eqnarray*}}
\def\:{:\,}

\newtheorem{theorem}{Theorem}[section]

\newtheorem{lemma}[theorem]{Lemma}

\newtheorem{remark}[theorem]{\bf Remark}
\newtheorem{definition}[theorem]{\bf Definition}

\def\sec{\setcounter{equation}{0}}

\def\Poi{\textup{Poi}}
\def\Bin{\textup{Bin}}

\def\ker{\textup{ker}}
\def\coker{\textup{coker}}
\def\im{\textup{im}}
\def\Cech{\v{C}ech\ }
\def\C{{\cal C}}
\def\cCB{{\cal C}_B}

\def\Rar{\Rightarrow}
\def\rar{\rightarrow}

\newcommand{\be}{\widehat{\beta}}

\newcommand{\ep}{\varepsilon}

\newcommand{\lam}{\lambda}

\def\mR{\mathbb{R}}

\def\mZ{\mathbb{Z}}

\def\mN{\mathbb{N}}

\def\mM{\mathbb{M}}

\def\definedas{\stackrel{\Delta}{=}}

\def\definedas{\stackrel{\Delta}{=}}

\def\convas{\stackrel{a.s. }{\to}}

\newcommand{\md}{{\,d}}

\newcommand{\cP}{\mathcal{P}}
\newcommand{\cU}{{\mathcal U}}
\newcommand{\cC}{{\mathcal C}}
\newcommand{\cX}{{\mathcal X}}

\newcommand{\cF}{{\mathcal F}}
\newcommand{\cK}{{\mathcal K}}
\newcommand{\cL}{{\mathcal L}}

\def\1{\mathbf{1}}

\def\smallhalf{\mbox{$\frac{1}{2}$}}
\def\smallquarter{\mbox{$\frac{1}{4}$}}
\begin{document}
	
	
	
	\title{Random geometric complexes in the thermodynamic regime 	}

	
	\author{D.\ Yogeshwaran\footnote{Stat. Math. Unit, Indian Statistical Institute, Bangalore, India. Research supported in part by FP7-ICT-318493-STREP.}, Eliran Subag\footnote{Mathematics and Computer Science, Weizmann Institute, Rehovot, Israel. Research supported in part by Israel Science Foundation, 853/10.}, and Robert J.\ Adler\footnote{Electrical Engineering, Technion, Haifa, Israel.
			Research supported in part by AFOSR FA8655-11-1-3039 and ERC  2012 Advanced Grant 20120216.}}
	
	\maketitle
	
	\begin{abstract}
		{\footnotesize
			We consider the topology of simplicial complexes with vertices the points of a random point process and faces determined by distance relationships between the vertices. In particular, we study the Betti numbers of these complexes as the number of vertices becomes large, obtaining limit theorems for means, strong laws, concentration inequalities and central limit theorems.
			
			As opposed to most  prior papers treating random complexes,  the limit with which we work is in the so-called `thermodynamic'  regime (which includes the percolation threshold) in which the complexes become very large and complicated, with complex homology characterised by diverging Betti numbers. The proofs combine  probabilistic arguments from the theory of stabilizing functionals  of point processes and  topological arguments exploiting the properties of Mayer-Vietoris exact sequences. The   Mayer-Vietoris arguments are crucial, since homology in general, and Betti numbers in particular, are global rather than local phenomena, and most standard probabilistic arguments are based on the additivity of  functionals arising as a consequence of locality.}
		
	\end{abstract}
	
	\ams{ Primary: 60B99 
		60D05, 
		05E45; 
		Secondary: 60F05, 
		55U10. 
	}
	
	\keywords{Point processes, Boolean model, random geometric complexes, limit theorems.}

	\section{Introduction}
	\label{sec:intro}
	\sec
	
	This paper is concerned with structures created by taking (many) random points and building the structure based on neighbourhood relations between the points.  Perhaps the simplest way to describe this is to let $\Phi = \{x_1,x_2,\dots\}$ be a finite or countable, locally finite, subset of points in $\mR^d$, for some $d>1$, and to consider the  set
	\begin{equation}
	\label{defn:boolean}
	\cCB(\Phi,r) \definedas \bigcup_{x \in \Phi}B_{x}(r),
	\end{equation}
	where $0<r<\infty$, and $B_x(r)$ denotes the $d$-dimensional ball of radius $r$ centred at $x \in \mR^d$.
	
	When the points of $\Phi$ are those of a stationary Poisson process on $\mR^d$, this union is a special case of a
	`Boolean model', and its integral geometric properties  -- such as volume, surface area, Minkowski functionals -- have been studied in the setting of stochastic geometry since the earliest days of that subject. Our interest, however, lies in the homological structure of $\cCB(\Phi,r)$, in particular, as expressed through its Betti numbers. Thus our approach will be via the tools of algebraic topology, and, to facilitate this, we shall generally work not with $\cCB(\Phi,r)$ but with a homotopically
	equivalent  abstract simplical complex with a natural combinatorial structure. This will be the \Cech complex with radius $r$ built over the point set $\Phi$,  denoted by $\cC(\Phi,r)$,  and defined below in Section \ref{sec:topology}.
	
	The first, and perhaps most natural topological question to ask about these sets is how connected are they.  This is more a graph theoretic    question than a topological one, and has been well studied in this setting, with \cite{Penrose03} being the standard text in the area. There are various `regimes' in which it is natural to study these questions,
	depending on the radius $r$. If $r$ is small, then the balls in \eqref{defn:boolean} will only rarely overlap, and so
	the topology of  both $\cCB(\Phi,r)$  and  $\cC(\Phi,r)$ will be  mainly that of many isolated points. This is known as the `dust regime'. However, as $r$ grows,
	the balls will tend to overlap, and so a large, complex structure will form, leading to the notion of `continuum percolation', for which the standard references are  \cite{Hall88} and \cite{Meester96}. The percolation transition occurs within what  is known as the  `thermodynamic', regime (described in more detail in Section \ref{sec:point_process}), and is typically the hardest to analyse. The third and final regime arises as $r$ continues to grow, and (loosely speaking)
	$\cCB(\Phi,r)$  merges into a single large set with no empty subsets and so no interesting topology.
	
	Motivated mainly by issues in
	topological data analysis (e.g.\  \cite{Niyogi08,Niyogi11}) there has been considerable recent interest
	in the topological properties   of $\cCB(\Phi,r)$  and  $\cC(\Phi,r)$   that go beyond mere connectivity or the volumetric measures provided by integral geometry. These studies were
	initiated by Matthew Kahle in \cite{Kahle11}, in a paper which studied the growth of the expected Betti numbers of these sets when the underlying
	point process $\Phi$ was either a Poisson process or a random sample from a distribution satisfying mild regularity properties.
	Shortly afterwards, more sophisticated distributional results were proven in   \cite{Kahle11a}.
	An extension  to more general stationary point processes $\Phi$ on $\mR^d$ can be found in \cite{Adler12},  while, in the Poisson
	and binomial settings,  \cite{Bobrowski11} looks at these problems from the point of view of the Morse theory of the distance function.
	Recently   \cite{Bobrowski13} has established  important -- from the point of view of applications --  extensions to the results of \cite{Bobrowski11,Kahle11,Kahle11a}  in which the underlying point process lies on a manifold of lower dimension than an ambient Euclidean space in which the balls of
	\eqref{defn:boolean} are defined. See also the recent survey \cite{Bobrowski14}.

	However, virtually all of the results described in the previous paragraph (with the notable exception of some growth results for expected  Betti numbers in \cite{Kahle11} and numbers of critical points in  \cite{Bobrowski11}) deal with the topology of the dust regime. What is new in the current paper is a focus on the thermodynamic regime, and  new results that go beyond the earlier ones about   expectations.  Moreover, because of the long range dependencies in the thermodynamic regime, proofs here involve considerably more topological arguments than is the case for the dust regime.
	
	Our main results are summarised in the following subsection, after which we shall give some more details about the current literature.
	Then, in Section \ref{sec:prelims}, we shall recall some basic notions from topology and from the theory of point processes. The new results begin in Section \ref{sec:stationary_pp}, where we shall treat the setting of general  stationary point processes, while the Poisson and binomial  settings will be treated in Section \ref{sec:Poisson_Binomial}.
	\remove{We shall conclude in Section \ref{sec:future} with some general remarks and open
		problems.}
	The paper concludes with some appendices containing variations  of some known tools, adapted to our needs.
	
	\subsection{Summary of Results}
	\label{sec:summary}
	Throughout the paper we shall assume that all our point processes are defined over $\mR^d$ for  $d \geq 2$. Denoting Betti numbers of a set $A \subset \mR^d$ by $\beta_k(A)$, $k=1,\dots,d-1$, we are interested in $\beta_k(\cCB(\Phi,r))$ for point processes $\Phi \subset \mR^d$.   Since  the Betti numbers for $k\geq d$ are identically zero, these values of $k$ are uninteresting. On the other hand, $\beta_0(A)$ gives the number of connected components of $A$. While this is clearly interesting and important in our setting, it has already been studied in detail from the point of view of random graph theory, as described above.  Indeed,  (sometimes stronger)  versions of virtually all our results for the higher Betti numbers already exist  for $\beta_0$ (cf.\ \cite{Akcoglu81,Penrose03}), and so this case will appear only peripherally  in what follows.
	
	Here is a  summary of our results,  grouped according to the underlying point processes involved. Formal definitions of technical terms are  postponed to later sections.

	\begin{enumerate}
		\item {\it General stationary point processes:} For a stationary point process $\Phi$ and $r \in (0,\infty)$, we study the asymptotics of $\beta_k(\cCB(\Phi \cap W_l,r))$ as $l \to \infty$ and where $W_l = [-\frac{l}{2},\frac{l}{2})^d$. We show convergence of expectations (Lemma \ref{lem:convergence_exp}) and, assuming ergodicity, we prove strong laws
		(Theorem \ref{thm:strong_law}) for all the Betti numbers and a concentration inequality for $\beta_0$ (Theorem
		\ref{thm:conc_ineq_b0_DPP}) in the special case of determinantal point processes.
		
		\item {\it Stationary Poisson point processes:} Retain the same notation as above, but take $\Phi = \cP$, a stationary Poisson point process on $\mR^d$. In this setting we prove a central limit theorem (Theorem \ref{thm:clt_Poisson_Binomial}) for the Betti numbers of  $\cCB(\cP\cap W_l, r)$  and  $\cC(\cP\cap W_l,r)$, for any $r\in(0,\infty)$,  as $l\to\infty$.  We also treat the case in which $l$  points are chosen uniformly  in $W_l$ and obtain a similar result, although in this case  we  can only prove the central limit theorem for $r \notin I_d$, where the interval $I_d$ will be defined in Section \ref{sec:clt}. Informally, $I_d$ is the interval of radii where both $\cCB(\cP,r)$ and its complement have unbounded components a.s.. We only remark here that $I_2 = \emptyset$ and $I_d$ is a non-degenerate interval for $d \geq 3$.
		
		\item {\it Inhomogeneous Poisson and binomial point processes:}  Now, consider either the Poisson point process $\cP_n$  with
		non-constant intensity function $nf$, for a `nice', compactly supported,  density $f$, or the binomial   process of $n$ iid random variables with probability density  $f$.   In this case the basic  set-up requires a slight modification, and so we consider asymptotics for $\beta_k(\cCB(\cP_n,r_n))$ as $n \to \infty$ and $nr_n^d \to r \in (0,\infty)$. We derive an upper bound for variances and a weak law (Lemma \ref{lem:weak_law_Poisson}). In the Poisson case, we also derive a variance lower bound for the top homology. For the corresponding binomial  case we prove a concentration inequality (Theorem \ref{thm:conc_ineq_Binomial}) and use this to prove a strong law for both cases (Theorem \ref{thm:strong_law_Poisson}).
	\end{enumerate}
	%
	
	
	A few words on our proofs: In the case of stationary point processes, we shall use the nearly-additive properties of  Betti numbers along with sub-additive theory arguments (\cite{Yukich98,Steele97}). In the Poisson and binomial  cases, the proofs center around an analysis of the so-called  add-one cost function,
	\beqq
	\beta_k(\cCB(\cP \cup \{O\},r)) - \beta_k(\cCB(\cP,r)),
	\eeqq
	where $O$ is the origin in $\mR^d$.
	While simple combinatorial topology bounds with martingale techniques suffice for strong laws, weak laws, and concentration inequalities, a more careful analysis via the Mayer-Vietoris sequence is required for  the central limit theorems.
	
	Our central limit theorems rely  on similar results for stabilizing Poisson functionals (cf.\ \cite{Penrose01}),  which in turn  were based upon martingale central limit theory.  As for  variance  bounds, while upper bounds can be derived via Poincar\'{e} or Efron-Stein inequalities, the more involved lower bounds exploit the recent bounds developed  in \cite{Last14} using chaos expansions of Poisson functionals.
	
	One of the difficulties in analyzing Betti numbers that will become obvious in the proof of the central limit theorem is their global nature. Most known examples of stochastic geometric functionals satisfy both the notions of stabilization (cf.\ \cite{Penrose01}) known as `weak' and `strong' stabilization. However, we shall prove that higher Betti numbers satisfy weak stabilization but satisfy strong stabilization only for certain radii regimes. We are unable to prove strong stabilization of higher Betti numbers for all radii regimes because of the global dependence of Betti numbers on the underlying point process.

	\subsection{Some history}
	\label{sec:related_work}
	
	To put our results into perspective, and to provide some motivation, here is a little history.
	
	As already mentioned,  recent interest in random geometric complexes was stimulated by their connections to topological data analysis and, more broadly, applied topology. There are a number of accessible surveys on this subject (e.g.\ \cite{Carlsson09,Costa12,Edelsbrunner10,Ghrist08,Zomorodian09}), all of which share a common theme of studying  topological invariants of simplicial complexes built on point sets. At the time of writing, another excellent review \cite{Carlsson14} by Carlsson appeared, which is longer than the earlier ones, more up to date, and which also contains a gentle introduction the topological concepts needed in the current paper.   Throughout this literature, Betti numbers, apart from being a simple topological invariant,  appear as the first step to understanding persistent homology, undoubtedly the single most important tool to emerge from current research in applied topology.
	
	Although  the study of random geometric complexes seems to have originated in \cite{Kahle11}, it is worth noting that Betti numbers of a random complex model were already investigated in \cite{Linial06},  where a higher-dimensional version of the  Erd\"{o}s-Renyi random graph model was constructed. The recent paper \cite{Kahle13} gives a useful survey of what is known about the topology of these models, which are rather different to those in the current paper.
	
	As we already noted above, the Boolean model \eqref{defn:boolean}   has long been studied in stochastic geometry, mainly through  its volumetric measures. However, one of these measures is the Euler characteristic,  $\chi$, which is also one of the basic homotopy invariants of topology,  and is given by an alternating sum of Betti numbers.
	Results for Euler characteristics which are related to ours for the individual Betti numbers can be found, for example, in \cite{Schneider08}, which establishes ergodic theorems for
	$\chi(\cCB(\Phi,r))$ when the underlying  point process $\Phi$ is itself ergodic. More recently, a slew of results have been established for  $\chi(\cCB(\cP,r))$ (i.e.\  the Poisson case)
	in the   preprint \cite{Hug13}. The arguments in this paper replace more classic integral geometric arguments, and are based on new results in the  Malliavin-Stein  approach to limit theory (cf.\  \cite{Nourdin12} and esp.\ \cite{Schulte13}). To some extent we shall also exploit these methods in the current paper, although they are not as well suited to the study of Betti numbers as they are to the Euler characteristic, due to the non-additivity of the former.
	
	An alternative approach to the Euler characteristic of a simplicial complex is via an alternating sum of the numbers of faces of different dimensions. This fact has been used to good effect in  \cite{Decreusefond11}, which derives exact expressions for moments of face counts,  and a central limit theorem and  concentration inequality for the Euler characteristic and
	$\beta_0$ when the underlying space is a torus. (Working on a torus rather avoids otherwise problematic boundary issues which complicate moment calculations.) Some additional results on phase transitions in face counts for a wide variety of  underlying stationary point processes can be found in \cite[Section 3]{Adler12}.

	\subsection{Beyond the \Cech complex}
	Although this paper concentrates on the \Cech complex as the basic topological object determined by a point process, this is but one of the many geometric complexes that could have been chosen. There are various other natural choices  including the  Vietoris-Rips, alpha, witness, cubical, and discrete Morse complexes (cf.\ \cite[Section 7]{Forman02},  \cite[Section 3]{Zomorodian12}) that  are also of  interest.  In particular, the alpha complex is homotopy equivalent to the \Cech complex (\cite[Section 3.2]{Zomorodian12}), as is an appropriate discrete Morse complex (\cite[Theorem 2.5]{Forman02}). This immediately implies that all the limit theorems for Betti numbers in this paper also hold for these complexes.
	
	Moreover, since our main topological tools  -- Lemmas \ref{lem:MV_complex_bounds} and \ref{lem:MV_complexes} --  can be shown to hold for all the complexes listed above, most of our arguments should easily extend to obtain similar theorems for these cases as well.
	
	\section{Preliminaries}
	\label{sec:prelims}
	\sec

	This section introduces a handful of basic concepts  and definitions from algebraic  topology  and the theory of point processes. The aim is  not to make the paper self-contained, which would be impossible, but to allow readers from these two areas to  have at least
	the vocabulary for reading our results. We refer readers to the standard texts  such as \cite{Hatcher02,Munkres84} for more details on the topology we need,
	while  \cite{Schneider08,Stoyan95} covers the point process material.
	
	\subsection{Topological Preliminaries}
	\label{sec:topology}
	
	An {\it abstract simplicial complex}, or simply  {\it complex},
	$\cK$ is a finite collection of finite sets such that $\sigma_1 \in \cK$ and $\sigma_2 \subset \sigma_1$ implies  $\sigma_2 \in \cK$. The sets in $\cK$ are called faces or simplices and the dimension $\text{dim}(\sigma)$ of any simplex $\sigma \in \cK$ is the cardinality of $\sigma$ minus $1$. If $\sigma \in \cK$ has dimension $k$, we say that $\sigma $ is a $k$-simplex of $\cK$.
	The $k$-skeleton of $\cK$, denoted by $\cK^k$, is the complex formed by all faces of $\cK$ with dimension at most $k$.
	
	Note that a singleton containing a simplex of dimension greater than zero is not necessarily a simplicial complex.  (This is as opposed to their usual concrete representations as subsets  of Eulcidean space, in which a closed simplex physically contains all its lower dimensional faces.) When we want to study the complex generated by a simplex $\sigma$, we shall refer to it as the {\it full simplex} $\sigma$, or, equivalently, $\sigma^k$, its $k$-skeleton, where $k =$ dim$(\sigma)$.
	
	A map $g\: \cK^0 \to \cL^0$ between two complexes $\cK$ and $\cL$ is said to be a {\em simplicial map} if, for any $m \geq 0$, $\{g(v_1),\ldots,g(v_m)\}\in\cal L$
	whenever $\{v_1,\ldots,v_m\}\in \cal K$.
	
	Given a point set in $\mR^d$ (or generally, in a metric space) there are various ways to define a simplicial complex that captures some of the geometry and topology related to the set. We shall be concerned with a specific construction -- the so-called \Cech complex.
	\begin{definition}
		\label{defn:cech}
		Let $\cX=\{x_i\}_{i=1}^n\subset\mR^d$ be a finite set of points. For any $r>0$, the \Cech complex of radius $r$
		is the abstract simplicial complex
		$$\C (\cX ,r) \triangleq \Big\{ \sigma\subset\cX\: \bigcap_{x\in\sigma} B_x(r)\neq\emptyset \Big\},$$
		where $B_x(r)$ denotes the ball of radius $r$ centered at $x$.
	\end{definition}
	For future reference, note that by the nerve theorem (cf.\ \cite[Theorem 10.7]{Bjorner95}) the \Cech complex built over a finite set of points is homotopic to the Boolean model (\ref{defn:boolean}) constructed over the same set.
	
	The \Cech complexes that we shall treat  will be be generated by random point sets,  and we shall be interested in their homology groups $H_k$, with coefficients from a field $\mathbb{F}$, which will be anonymous but fixed throughout the paper.   A common choice is
	to take $\mathbb{F}=\mathbb{Z}_2$, which is computationally convenient, but this will not be necessary here.
	
	A few words are in place for the reader unfamiliar with homology theory. On the heuristic level, the homology groups of a space are meant to capture the topological structure of cycles or `holes' in it. Of course, such concepts are best understood in a geometric setting,
	e.g.\ when the space is a subset in $\mathbb{R}^{d}$, or an abstract complex generated by a triangulation of such a subset. Nevertheless, high dimensional cycles can be defined combinatorially, much like $1$-dimensional ones are defined for graphs. Besides simply defining the cycles, one wishes to be able to ignore trivial cycles or ones that are equivalent to others. As concrete examples, the boundary of a full disc should not be regarded as a `hole' and the two cycles forming the boundary of a hollow cylinder are to be thought of as representatives of the same `hole' (where, to obtain an abstract complex, one should consider triangulations of these objects, or alternatively work with singular homology). The way the theory deals with these two issues is by defining $H_{k}$
	as the \emph{quotient} of a group $Z_{k}$ representing cycles by another group $B_{k}$  representing boundaries. Then, trivial cycles are exactly those in the class of $0$
	and equivalent ones belong to the same class. The groups $Z_{k}$
	and $B_{k}$
	are subgroups of the free group generated by (oriented) simplices; i.e.\ their elements are formal sums of simplices with coefficients taken from some field. In general, the coefficients are from an Abelian group but we shall work with field coefficients. Having made the choice of working with field coefficients, all groups in our case are vector spaces. The dimension of $H_{k}$, denoted by $\beta_k$,
	is called the $k$-th Betti number and has a special meaning: it is the maximal number of \emph{non-equivalent} cycles of dimension $k$. It is important to note that, for $k=0$, $\beta_0$ is the maximal number of vertices which (pairwise) cannot be connected by a sequence of $1$-simplices; that is, $\beta_0$ is the number of connected components of the space. Throughout the paper, we shall concentrate on the (random) Betti numbers $\beta_k$, $0\leq k\leq d-1$, of \Cech complexes.
	%
	%
	
	
	Our two main topological tools are collected in the following two lemmas. The first  is needed for obtaining various moment bounds on Betti numbers of random simplicial complexes, and the second will replace the role that additivity of functionals usually plays in most probabilistic limit theorems. Because the arguments underlying these lemmas  are important for what follows, and will be unfamiliar to most probabilistic readers, we shall  prove them both.  However both  contain results  that are well known to topologists.

	\begin{lemma}
		\label{lem:MV_complex_bounds}
		Let $\cK,\cK_1$ be two finite simplicial complexes such that $\cK \subset \cK_1$ (i.e., every simplex in $\cK$ is also a simplex in $\cK_1$). Then, for every $k \geq 1$, we have that
		\[ \big|\beta_k(\cK_1) - \beta_k(\cK)\big| \ \leq \  \sum_{j=k}^{k+1} \mbox{$\# \big\{ j${\rm -simplices in }$\cK_1\setminus\cK\big\}$}. \]
	\end{lemma}
	\begin{proof} We start with the simple case when $\cK_1 = \cK \bigcup \{\sigma\}$ where $\sigma$ is a $j$-simplex for some $j \geq 0$.  Note
		that since both  $\cK$ and  $\cK_1$ are simplicial complexes it follows that  all the proper subsets of $\sigma$ must already be
		present in $\cK$. Thus we immediately have that
		\begin{equation*}
		\beta_k(\cK_1) - \beta_k(\cK) \ \in \
		\begin{cases}
		\{0\} & \text{$j \neq k,k+1$,} \\
		\{0,1\} & \text{$j = k$,} \\
		\{-1,0\} & \text{$j = k+1$.}
		\end{cases}
		\end{equation*}
		Thus the lemma is proven for the case $\cK_1 = \cK \bigcup \{\sigma\}$.  For arbitrary complexes $\cK \subset \cK_1$, enumerate the simplices in $\cK_1\setminus\cK$ such that lower dimensional simplices are added before the higher dimensional ones and repeatedly apply the above argument along with the triangle inequality.
	\end{proof}
	
	With a little more work, one can go further than the previous lemma and derive an explicit equality for differences of Betti numbers. This is again a classical result in algebraic topology which is derived using the Mayer-Vietoris sequence (see \cite[Corollary 2.2]{Delfinado93}). However we shall state it here as it is important for our proof of the central limit theorem.
	
	A little notation is needed before we state the lemma. A sequence of Abelian groups $G_1,\ldots,G_l$ and homomorphisms $\eta_i: G_i \to G_{i+1}$,  $i=1,\ldots,l-1$ is said to be {\em exact} if $\im \,\eta_i = \ker \, \eta_{i+1}$ for all $i = 1,\ldots,l-1$. If $l = 5$ and $G_1$ and $G_5$ are trivial, then  the sequence is called {\em short exact}.
	\begin{lemma}[Mayer-Vietoris Sequence]
		\label{lem:MV_complexes}
		Let $\cK_1$ and $\cK_2$ be two finite simplicial complexes and $\cL = \cK_1 \cap \cK_2$ (i.e., $\cL$ is the complex formed from all the simplices in both $\cK_1$ and $\cK_2$). Then the following are true:
		\begin{enumerate}
			\item The following is an exact sequence, and, furthermore, the homomorphisms $\lambda_k$ are  induced by inclusions:
			\begin{align*}
			& \cdots\to H_k(\cL) \stackrel{\lambda_k}{\to} H_k(\cK_1)\oplus H_k(\cK_2) \to H_k(\cK_1 \cup \cK_2)
			\\ &\qquad\qquad\qquad \qquad\qquad\qquad
			\to H_{k-1}(\cL) \stackrel{\lambda_{k-1}}{\to} H_{k-1}(\cK_1)\oplus H_{k-1}(\cK_2)\to \cdots
			\end{align*}

			\item Furthermore,
			\[ \beta_k(\cK_1 \bigcup \cK_2) \ = \ \beta_k(\cK_1) + \beta_k(\cK_2)+ \beta(N_k) +\beta(N_{k-1})-\beta_k(\cL), \]
			where $\beta(G)$ denotes the rank of a vector space $G$ and $N_j = \ker \, \lambda_j$.
			
		\end{enumerate}
	\end{lemma}

	\begin{proof} The first part of the lemma is just a simplicial version of the classical Mayer-Vietoris theorem (cf.\   \cite[Theorem 25.1]{ Munkres84}). The second part follows from the first part, as follows: Suppose we have the  exact sequence
		\[ \cdots \to G_1 \stackrel{\eta_1}{\to} G_2 \stackrel{\eta_2}{\to} G_3 \stackrel{\eta_3}{\to} G_4 \stackrel{\eta_4}{\to} G_5 \to \cdots \]
		Then we also have the  short exact sequence
		\[ 0 \to \coker \, \eta_1 \to G_3 \to \ker \, \eta_4 \to 0, \]
		where the quotient space  $\coker \, \eta_1 = G_2 / \im \, \eta_1$ is the cokernel  of $\eta_1$ . From the exactness of the sequence we have that
		\[ \beta(G_3) = \beta(\coker \, \eta_1) + \beta(\ker \, \eta_4). \]
		Now applying this to the Mayer-Vietoris sequence with $G_1 = H_k(\cL)$, etc, we have
		\begin{eqnarray*}
			\beta_k(\cK_1 \bigcup \cK_2) & = & \beta(\coker \, \lam_k) + \beta(\ker \, \lam_{k-1}) \\
			& = & \beta_k(\cK_1) + \beta_k(\cK_2) - \beta(\im \, \lam_k) + \beta(N_{k-1}) \\
			& = & \beta_k(\cK_1) + \beta_k(\cK_2)+ \beta(N_k) +\beta(N_{k-1})-\beta_k(\cL),
		\end{eqnarray*}
		which completes the proof.
	\end{proof}
	
	\subsection{Point Process Preliminaries}
	\label{sec:point_process}
	
	A point process $\Phi$ is formally defined to be a random, locally-finite (Radon), counting measure on $\mR^d$.
	More formally, let $\mathcal{B}_{b}$ be   the  $\sigma$-ring of bounded, Borel subsets of $\mR^d$ and let $\mM$
	be the corresponding space of non-negative Radon counting measures. The Borel $\sigma$-algebra
	$\cal M$ is generated by the mappings $\mu\to\mu(B)$  for all   $B\in \mathcal{B}_{b}$. A point process    $\Phi$ is a
	random element in $(\mM, \cal M)$, i.e.\ a measurable map from a probability space
	$(\Omega,\cal F,\mP)$
	to $(\mM,\cal M)$. The distribution of   $\Phi$ is the measure $\mP\Phi^{-1}$ on $(\mM,\cal M)$.

	We shall
	typically identify $\Phi$ with the positions $\{x_1,x_2,\dots\}$ of its atoms, and so for Borel $B\subset\mR^d$, we shall allow ourselves to write
	\beqq
	\Phi(B) \ = \ \sum _i\delta_{x_i}(B) \ = \ \#\{i\:x_i\in B\} \ = \ \#\{\Phi\cap B\},
	\eeqq
	where $\#$ denotes cardinality and $\delta_x$ the single atom measure with mass one at $x$.
	The intensity measure of $\Phi$ is the non-random measure defined by $\mu(B)=\EXP{\Phi(B)}$, and, when $\mu$ is absolutely continuous with respect to Lebesgue measure, the corresponding density is  called the intensity of $\Phi$. A point process is called {\em simple} if its points (i.e., $x_i$'s) are a.s.\ distinct. In this article, we shall consider only simple point processes.
	
	For a measure $\phi \in \mM$, let $\phi_{(x)}$ be the translate measure given by
	$\phi_{(x)}(B) = \phi(B-x)$ for $x \in \mR^d$ and $B\in \mathcal{B}_{b}$. A point process is said to be \emph{stationary} if the distribution of $\Phi_{(x)}$ is invariant under such translation, i.e.\
	$\P \Phi^{-1}_{(x)} = \P \Phi^{-1}$ for all $x \in \mR^d$.
	For a stationary point process in $\mR^d$, $\mu(B) = \lam |B|$ for all $B\in \mathcal{B}_b$, where $|B|$ denotes the Lebesgue measure of $B$, and the constant
	of proportionality  $\lam$ is called the {\em intensity} of the point process.
	
	Of particular importance to us are the Poisson and Binomial point processes. These processes are characterized through their relation to one of the most fundamental notions of probability theory - statistical independence. A Poisson process $\cP$ is the simple point process uniquely determined by its intensity measure $\mu$ and the following property: for any collection of disjoint measurable sets $\{A_i\}$, $\{\cP(A_i)\}$ are independent random variables. An equivalent, direct definition is given by the finite dimensional distributions,
	\[
	\pr{\cP(A_i)=n_i,i=1,...,k}=\prod^k_{i=1}\pr{P_i=n_i},
	\]
	where $P_i$ are Poisson variables with parameter $\mu(A_i)$ and, again, $A_i$ are assumed to be disjoint. A Binomial point process  $\cX_n$ is a process formed by $n$ i.i.d points $X_1,...,X_n$. It is worth mentioning that conditioning a Poisson process to have exactly $n$ points yields a Binomial process; and conversely, mixing a Binomial process by taking $n$ to be a Poisson variable produces a Poisson process.

	For all of the point processes we consider, we shall be interested in  behavior in the so-called thermodynamic limit. That is, while letting the number of points $n$ increase to infinity, we choose the radii $r_n$ so that the average degree of a point (in the random $1$-skeleton) converges to a constant. (Note, however, that this average depends on the location of the point for inhomogeneous processes). As was described in Section \ref{sec:summary}, this is done either by scaling the space and taking $r$ to be $n$-independent for stationary processes, or by fixing the space, increasing the intensity and decreasing $r_n$ for inhomogeneous processes.

	We conclude  the section with some more definitions. For Borel $A \subset \mR^d$, we write $\Phi_A$ for both the restricted random measure given by $\Phi_A(B):=\Phi(A\cap B)$ (when treating $\Phi$ itself as a measure) and the point set $\Phi\cap A$ (when treating $\Phi$ as a point set).
	To save space, we shall write $\Phi_l$  for $\Phi_{W_l}$, where  $W_l$ is the `window' \ $[-l/2,l/2)^d$, for all $l \geq 0$.

	For a set of measures $ \Theta \in \cal{M}$, let the translate family be $\Theta_x :=
	\{\phi_{(x)} \: \phi \in \Theta \}$. A point process $\Phi$ is said to be {\em ergodic} if
	\beqq
	\pr{\Phi \in {\Theta}} \ \in \ \{0,1\}
	\eeqq
	for all $\Theta \in \mathcal{M}$ for which
	\beqq
	\pr{\Phi \in ({\Theta}\setminus\Theta_x) \,\cup\, ({\Theta}_x\setminus{\Theta})}\  =\   0
	\eeqq
	for all $x \in \mR^d$.

	Finally, we say that \emph{$\Phi$ has all moments} if, for all bounded Borel $B \subset \mR^d$, we have
	\begin{equation}
	\label{eqn:pp_moments}
	\EXP{\left[\Phi(B)\right]^k} < \infty, \quad  \text{for all} \ k \geq 1.
	\end{equation}

	\section{Limit theorems for stationary point processes}
	\label{sec:stationary_pp}
	\sec
	
	This section is concerned with the \Cech complex $\C (\Phi_l ,r)$, where $\Phi$ is a stationary point process on $\mR^d$ with unit intensity and, as above,  $\Phi_l$ is the restriction of $\Phi$ to the window $W_l=[-l/2,l/2)^d$. The radius $r$ is arbitrary but fixed.

	It is natural to expect that, as a consequence of stationarity, letting $l\to\infty$, $l^{-d}\EXP{\beta_k(\C (\Phi_l ,r))}$ will converge to a limit. Furthermore, if we also assume ergodicity for $\Phi$, one expects convergence of $l^{-d} \beta_k(\C (\Phi_l ,r))$ to a random
	limit.  All this would be rather standard fare, and rather easy to prove from general limit theorems, if it were only true that Betti numbers were additive functionals on simplicial complexes, or, alternatively, the Betti numbers of \Cech complexes were additive functionals of the underlying point processes.  Although this is not the case, Betti numbers are `nearly additive',  and a correct quantification of this near additivity is what will be required for our proofs.
	
	\remove{Since this is not the case, we need to proceed by replacing
		Betti numbers by a quantity that {\it is} additive, that is amenable to analysis, and that, at least asymptotically, will be equivalent to Betti numbers.}
	
	As hinted before Lemma \ref{lem:MV_complex_bounds}, the additivity properties of Betti numbers are related to simplicial counts $S_j (\cX,r)$, which, for $j\geq 0$,  denotes the number of $j$-simplices in $\C(\cX,r)$, and $S_j(\cX,r;A)$, which denotes the number of $j$-simplices with at least one vertex in $A$.
	%
	
	Our first results are therefore limit theorems for these quantities.
	
	\begin{lemma}
		\label{lem:upper_bound_simplices}
		Let $\Phi$ be a unit intensity stationary point process on $\mR^d$, possessing all moments. Then, for each $j\geq 0$, there exists a constant $c_j:= c(\cL_{\Phi},j,d,r)$
		such that

		\[ \EXP{S_j(\Phi_{A},r)} \ \leq \ \EXP{S_j(\Phi,r;A)}\  \leq \ c_j |A|. \]
	\end{lemma}
	
	\begin{proof} We have the following trivial upper bound for simplicial counts:
		\[ S_j(\Phi,r;A)  \ \leq \  \sum_{x \in \Phi \cap A}\left(\Phi(B_x(2r))\right)^{j-1}. \]
		Due to the stationarity of $\Phi$ along with the assumption that it has all moments, we have that the measure
		\[  \mu_0(A)\ := \ \EXP{\sum_{x \in \Phi \cap A}(\Phi(B_x(r))^{j-1}}\]
		is translation invariant and finite on compact sets.
		Thus $\mu_0(A) = c_j |A|$ for some $c_j\in (0,\infty)$,  and we are done.
	\end{proof}
	
	%
	%
	
	%
	\begin{lemma}
		\label{lem:strong_law_simplices}
		Let $\Phi$ be a unit intensity, ergodic, point process  on $\mR^d$ possessing all moments. Then, for each
		$j\geq 0$,  there exists a constant,  $\widehat{S}_j:=\widehat S (\cL_{\Phi},j,d,r)$, such that, with probability one,
		\[\lim_{l \to \infty}\frac{S_j(\Phi,r;W_l)}{l^d}
		\ =\  \lim_{l \to \infty}\frac{S_j(\Phi_l,r)}{l^d} \  =\  \widehat{S}_j(\cL_{\Phi},r).  \]
	\end{lemma}
	\begin{proof}  Define the function  
		\[ h(\Phi)\ := \ \frac{1}{j+1}  {\sum_{x \in \Phi_{W_1}}\mbox{\#[$j$-simplices in $\C(\Phi,r)$ containing $x$]}}.\]
		Recalling that by $\Phi-z$ we mean the points of $\Phi$ moved by $-z$,  it is easy to check that
		%
		\beq
		\label{bounds1}
		\sum_{z \, \in\,  \mZ^d \cap W_{l-2r-1}} h(\Phi - z) \ \leq \
		S_j(\Phi_l,r) \  \leq\    \sum_{z\, \in\,  \mZ^d \cap W_{l+1}} h(\Phi - z) .
		\eeq
		%
		Since $\Phi$ has all moments, we have that
		\[
		\EXP{h(\Phi)} \ \leq \ \EXP{\Phi(W_{1+r})^{j+1}} \  <\  \infty.
		\]
		and so are in position to apply the multivariate ergodic theorem (e.g.\ \cite[Proposition 2.2]{Meester96}) to
		each of  the sums in \eqref{bounds1}. This implies the existence of a constant $\widehat{S}_j(\cL_{\Phi},r) \in [0,\infty)$ such that, with probability one,
		\[ \lim_{l \to \infty} \frac{1}{l^d}\sum_{z \, \in\,  \mZ^d \cap W_{l-2r-1}} h(\Phi - z) \ = \
		\lim_{l \to \infty} \frac{1}{l^d}\sum_{z \, \in\, \mZ^d \cap W_{l+1}} h(\Phi - z)\
		= \    \widehat{S}_j(\cL_{\Phi},r).\]
		This gives the ergodic theorem for $S_j(\Phi_l,r)$. The result for $S_j(\Phi,r;W_l)$ follows from this and  the  bounds
		\[ S_j(\Phi_l,r) \  \leq\    S_j(\Phi_l,r ; W_{l})\  \leq\   S_j(\Phi_{l+2r+1},r).\]
	\end{proof}
	
	\subsection{Strong Law for Betti numbers}
	
	
	In this section we shall start with  a convergence  result for  the expectation of $\beta_k(\C(\Phi_l,r))$  when $\Phi$ is a quite general stationary point process, and then proceed to a strong law. We treat these results separately, since convergence of expectations can be obtained under weaker conditions than the strong law. In addition, seeing the proof for  expectations first should make the proof of strong law easier to follow.
	
	From \cite[Theorem 4.2]{Adler12} we know that
	\beqq
	\EXP{\beta_k(\C(\Phi_l,r))} \ = \ O(l^d).
	\eeqq
	The following lemma strengthens this result.
	\begin{lemma}
		\label{lem:convergence_exp}
		Let $\Phi$ be a unit intensity stationary point process possessing all moments. Then, for each $0\leq k\leq d-1$,  there exists a  constant $\be_k:= \be_k(\cL_{\Phi},r) \in [0,\infty)$ such that
		\[
		\lim_{l\to\infty} \frac{\EXP{\beta_k(\C(\Phi_l,r))}}{l^d}\ =\ \be_k. \]
	\end{lemma}
	
	\begin{remark}
		The lemma is interesting  only in the case when $\be_k > 0$, and this does not always hold. However, it can be guaranteed for negatively associated point processes (including Poisson processes, simple perturbed lattices and determinantal point processes) under some simple conditions on void probabilities, cf.\ \cite[Theorem 3.3]{Adler12}.
	\end{remark}

	\begin{proof}[Proof of Lemma \ref{lem:convergence_exp}]
		
		Set
		\beqq
		\psi(l)\ := \ \EXP{\beta_k(\C(\Phi_l,r))},
		\eeqq
		and define
		\begin{equation}
		\label{eqn:limsup_betti}
		\be_k\ :=\   \limsup_{l \to \infty}\frac{\psi(l)}{l^d}.
		\end{equation}

		Fix $t > 0$. Let $Q_{it}, i = 1,\ldots,m^d$ be an enumeration of $\{tz_i + W_t \subset W_{mt}\: z_i \in \mZ^d\}$. Note that the $Q_{it}, i = 1,\ldots,m^d$ form a partition $W_{mt}$.
		
		Define the complex
		\beqq
		\cK(r,t)\ :=  \ \bigcup_{i=1}^{m^d}\C(\Phi_{Q_{it}},r),
		\eeqq
		and note that it is a subcomplex of $\C(\Phi_{mt},r)$. Since the union here is of disjoint complexest,
		$$ \beta_k(\cK(r,t))\  =\  \sum_{i=1}^{m^d} \beta_k(\C(\Phi_{Q_{it}},r)).$$
		
		Note that the vertices of any simplex in $\C(\Phi_{mt},r) \setminus \cK(r,t)$ must lie in the set $\bigcup_{i=1}^{m^d} (\partial Q_{it})^{(2r)}$, where for any set $A\subset \mR^d$, $A^{(r)}$ is the set of points in $\mR^d$ with distance at most $r$ from $A$.
		Hence, by Lemma \ref{lem:MV_complex_bounds},
		\begin{equation}
		\Big| \beta_k(\C(\Phi_{mt},r)) \, -\, \sum_{i=1}^{m^d} \beta_k(\C(\Phi_{Q_{it}},r)) \Big|
		\ \leq\   \sum_{j=k}^{k+1} S_j(\Phi_{\bigcup_{i=1}^{m^d} (\partial Q_{it})^{(2r)}},r). \label{eqn:ineq_betti_numbers}
		\end{equation}
		Thus, since for $c:=c(d,r)$ large enough, for any $t\geq 1$,
		
		\[\big\|\bigcup_{i=1}^{m^d} (\partial Q_{it})^{(2r)} \big\| \ \leq \ cm^dt^{d-1}, \]
		it follows from Lemma \ref{lem:upper_bound_simplices} that
		
		\begin{equation}
		\label{eqn:boundary_simplices}
		\frac{1}{(mt)^d} \,
		\EXP{ \sum_{j=k}^{k+1} S_j(\Phi_{\bigcup_{i=1}^{m^d} (\partial Q_{it})^{(2r)}},r)}
		\  \leq \ \frac{c}{t}.
		\end{equation}
		By the stationarity of $\Phi$, taking expectations over (\ref{eqn:ineq_betti_numbers}) and applying \eqref{eqn:boundary_simplices} we obtain that, for any $t\geq 1$,
		\[ \frac{\psi(mt)}{(mt)^d}\  \geq \ \frac{\psi(t)}{t^d} - \frac{c}{t}. \]

		Now fix $\ep >0$. By (\ref{eqn:limsup_betti}), we can find an arbitrarily large $t_0 \geq 1$ such that $\frac{\psi(t_0)}{t_0^d} \geq \be_k - \frac{\ep}{2}$ and $\frac{c}{t_0} \leq \frac{\ep}{2}$. Hence, from the above we have that, for all $m \geq 1$,
		\[ \frac{\psi(mt_0)}{(mt_0)^d} \ \geq \ \be_k - \ep. \]

		Now take  $l > 0$, and let $m$ be the unique integer $m=m(l)$ such that $mt_0 \leq l < (m+1)t_0$. Again, applying Lemma \ref{lem:MV_complex_bounds} yields
		\begin{eqnarray}
		\big|\beta_k(\C(\Phi_l,r)) - \beta_k(\C(\Phi_{mt_0},r))\big|
		& \leq & \sum_{j = k}^{k+1}S_j(\Phi_l,r ; W_l\setminus W_{mt_0}).\label{eqn:error_term_betti}
		\end{eqnarray}
		Since $\|W_l\setminus W_{mt_0}\|\leq d(l-mt_0)l^{d-1}$, as before, using Lemma \ref{lem:upper_bound_simplices}, it is easy to verify that
		\begin{eqnarray*}
			\frac{\psi(l)}{l^d} & \geq & \frac{\psi(mt_0)^d}{(m+1)^d t_0^d} - O(m^{-1}) \ \geq\ (\be_k - \ep)\frac{m^d}{(m+1)^d} - O(m^{-1}).
		\end{eqnarray*}
		Since $m \to \infty$ as $l \to \infty$, it follows that $\liminf_{l \to \infty} \frac{\psi(l)}{l^d} \geq \be_k - \ep.$ This and  (\ref{eqn:limsup_betti})  complete the proof.
	\end{proof}
	
	\begin{theorem}
		\label{thm:strong_law}
		Let $\Phi$ be a unit intensity ergodic point process possessing all moments. Then, for  $0\leq k\leq d-1$, and $\be_k$ as in Lemma \ref{lem:convergence_exp},
		\[\frac{\beta_k(\C(\Phi_l,r))}{l^d} \ \convas \  \be_k.
		\]
	\end{theorem}
	
	\begin{proof}
		As in the previous proof,  fix $t > 0$ and let $Q_{it}, i = 1,\ldots,m^d$  be the partition of
		$W_{mt}$ to translations of $W_t$.
		Further, for each real  $l>0$, let $m=m(l,t)$ be the  unique integer for which $mt \leq l < (m+1)t$.
		
		The proof contains two steps. Firstly, we shall establish a strong law for $\beta_k(\C(\Phi_{W_{mt}},r))$ in $m$, and then show that the error term in (\ref{eqn:error_term_betti}) vanishes asymptotically.  Many of our arguments will rely on the
		multi-parameter ergodic theorem (e.g\ \cite[Proposition 2.2]{Meester96}).

		Let $e_i,\ i=1,\ldots,d$, be the $d$ unit vectors in $\mR^d$, and $T_i=T_i(t)$  the measure preserving transformation  defined by a shift of $te_i$.
		Then, setting $Y = \beta_k(\C(\Phi_{W_t},r))$, and noting that   $\EXP{Y} \leq \EXP{\Phi(W_t)^{k+1}} < \infty$,   it follows immediately from the multi-parameter ergodic theorem that
		\beq
		\frac{1}{m^d}\sum_{i_1=0}^{m-1}\ldots \sum_{i_d=0}^{m-1}Y(T_1^{i_1}\ldots T_d^{i_d}(\Phi))
		&=& \frac{1}{m^d} \sum_{i=1}^{m^d} \beta_k(\C(\Phi_{Q_{it}},r))
		\notag \\
		&\convas &  \EXP{\beta_k(\C(\Phi_{W_t},r))},  \label{robconv1}\eeq
		as $m\to\infty$.
		
		Applying the multiparameter ergodic theorem again, but now with
		\beqq
		Y \ =\  S_k(\Phi_{(\partial W_t)^{(4r)}},r) + S_{k+1}(\Phi_{(\partial W_t)^{(4r)}},r),
		\eeqq
		we obtain
		\beq
		\label{robconv2}
		\frac{1}{m^d}\sum_{i=1}^{m^d} \left( S_k(\Phi_{(\partial Q_{it})^{(4r)}},r) + S_{k+1}(\Phi_{(\partial Q_{it})^{(4r)}},r) \right)\  \convas\ \widehat{S}_t,
		\eeq
		where $\widehat{S}_t \leq  ct^{d-1}$. This bound follows by applying Lemma \ref{lem:upper_bound_simplices} to obtain
		\beq
		\label{robeq1}
		\EXP{S_k(\Phi_{(\partial Q_{it})^{(4r)}},r) + S_{k+1}(\Phi_{(\partial Q_{it})^{(4r)}},r)} \leq  ct^{d-1},
		\eeq
		for some $c:= c(\cL_{\Phi},j,d,r)$.
		
		Note that, for $j=k,k+1$,
		\beq
		\label{robconv4}
		S_j(\Phi_{\bigcup_{i=1}^{m^d} (\partial Q_{it})^{(2r)}},r) \ \leq \ \sum_{i=1}^{m^d} S_j(\Phi_{(\partial Q_{it})}^{(4r)},r).
		\eeq
		It follows immediately from \eqref{robconv1}--\eqref{robconv4} that, with probability one,
		%
		\begin{align*}
		& \limsup_{m \to \infty} \Big| \frac{\beta_k(\C(\Phi_{mt},r))}{(mt)^d} - \frac{\EXP{\beta_k(\C(\Phi_{W_t},r))}}{t^d} \Big| \\
		&\, \, \, =\ \limsup_{m \to \infty} \Big| \frac{\beta_k(\C(\Phi_{mt},r))}{(mt)^d} - \frac{\sum_{i=1}^{m^d} \beta_k(\C(\Phi_{Q_{it}},r))}{(mt)^d}\Big| \\
		&\, \, \, \leq\ \limsup_{m \to \infty} \frac{1}{(mt)^d} \sum_{j=k}^{k+1} S_j(\Phi_{\bigcup_{i=1}^{m^d} (\partial Q_{it})^{(2r)}},r) \\
		&\, \, \, \leq\ \limsup_{m \to \infty} \frac{1}{(mt)^d} \sum_{j=k}^{k+1} \sum_{i=1}^{m^d} S_j(\Phi_{(\partial Q_{it})}^{(4r)},r)
		\\ &\, \, \, =\  \frac{\widehat{S}_t}{t^d}
		\  \leq\ \frac{c}{t}.
		\end{align*}
		
		Now, given $\varepsilon > 0$, by Lemma \ref{lem:convergence_exp}, we can choose $t_0$
		large enough so that, with probability one,
		\[ \lim_{m \to \infty} \Big| \frac{\beta_k(\C(\Phi_{mt_0},r))}{(mt_0)^d} - \be_k \Big|\ \leq\ \varepsilon. \]
		Now consider the error terms in (\ref{eqn:error_term_betti}). For $j = k,k+1$, we have that
		\beqq
		\frac{S_j(\Phi_l,r ; W_l\setminus W_{m(l)t_0})}{l^d} & \leq &\frac{S_j(\Phi_l,r)}{l^d} - \frac{S_j(\Phi_{m(l)t_0},r)}{l^d} \\
		& \leq& \frac{S_j(\Phi_{(m(l)+1)t_0},r)}{(m(l)t_0)^d} - \frac{S_j(\Phi_{m(l)t_0},r)}{((m(l)+1)t_0)^d}.
		\eeqq
		%
		%
		%
		By Lemma \ref{lem:strong_law_simplices}, we know that there exist
		$\widehat{S}_j(\Phi,r) \in [0,\infty)$, $j = k,k+1$, such that, with probability one,
		\beqq
		\lim_{l \to \infty} \frac{S_j(\Phi_{(m(l)+1)t_0},r)}{(m(l)t_0)^d} &=&  \lim_{l \to \infty} \frac{S_j(\Phi_{m(l)t_0},r)}{((m(l)+1)t_0)^d} \\
		&=& \widehat{S}_j(\Phi,r).
		\eeqq
		Hence, with probability one,
		\[ \lim_{l \to \infty} \frac{1}{l^d}\sum_{j = k}^{k+1}S_j(\Phi_l,r ; W_l\setminus W_{m(l)t_0})
		\ =\  0. \]
		Substituting this in (\ref{eqn:error_term_betti}) gives that, with probability one,
		\[ \lim_{l \to \infty} \Big|\frac{\beta_k(\C(\Phi_l,r))}{l^d} - \frac{\beta_k(\C(\Phi_{m(l)t_0},r))}{l^d}
		\Big|
		\ = \ 0,
		\]
		so that
		\[ \lim_{l \to \infty} \Big|\frac{\beta_k(\C(\Phi_l,r))}{l^d} - \be_k \Big| \ \leq\  \ep,  \]
		and the proof is complete.
	\end{proof}

	The following concentration inequality is an easy consequence of the general concentration inequality of \cite{Pemantle11}.
	\begin{theorem}
		\label{thm:conc_ineq_b0_DPP}
		Let $\Phi$ be a unit intensity stationary determinantal point process. Then for all $l \geq 1$,
		$\ep > 0$, and $a \in (\frac{1}{2},1]$, we have that
		\[ \pr{\Big|\beta_0(\C(\Phi_{l^{\frac{1}{d}}},r)) - \EXP{\beta_0(\C(\Phi_{l^{\frac{1}{d}}},r))})\Big| \,\geq\, \ep l^a} \ \leq\ 5 \exp\left(-\frac{\ep^2 l^{2a-1}}{16K_d(\ep l^{a-1} + 2K_d)} \right), \]
		where $K_d$ is the maximum number of disjoint unit balls that can be packed into $B_O(2)$.
	\end{theorem}
	\begin{proof} Firstly, note that $\beta_0$, viewed as a function on finite point sets is
		$K_d$-Lipschitz; viz.\ for any finite point set $\cX \subset \mR^d$ and $x \in \mR^d$,
		\[ \big|\beta_0(\C(\cX \cup \{x\},r)) - \beta_0(\C(\cX,r))\big| \ \leq \  K_d .\]
		This follows from the fact that, on the one hand,  adding a point $x$ to $\cX$ can add no more than one connected component to $\C(\cX,r)$.  On the other hand,  the largest decrease in the number of disjoint components in $\C(\cX,r)$  is bounded by the number of disjoint $r$-balls in $B_x(2r)$. By scale invariance, the latter number depends only on the dimension $d$ and not on $r$, and is denoted by $K_d$.
		
		The remainder of the proof is a simple application of \cite[Theorem 3.5]{Pemantle11} (see also \cite[Example 6.4]{Pemantle11}).
	\end{proof}

	\section{Poisson and binomial point processes}
	\label{sec:Poisson_Binomial}
	\sec
	
	Since there is already an extensive literature on $\beta_0(\C(\cX,r))$ for Poisson and binomial point processes, albeit in the language of connectedness of random graphs (e.g.\  \cite{Penrose03}),  in this section
	we shall restrict ourselves only to $\beta_k$ for $ 1\leq k\leq d-1$.
	
	The models we shall treat start with a Lebesgue-almost everywhere continuous probability density  $f$ on $\mR^d$,  with a compact, convex support that (for notational convenience) includes the origin, and such that
	\begin{equation}
	\label{condn:density}
	0 \  <\
	\inf_{x \in \text{supp}(f)}f(x)  \definedas f_* \ \leq\ f^* \definedas \sup_{x \in \mR^d} f(x) < \infty .\end{equation}
	The models are  $\cP_n$,  the Poisson point process on $\mR^d$ with intensity $nf$, and the  binomial point process $\cX_n=\{X_1,\ldots,X_n\}$, where the $X_i$ are i.i.d.\ random vectors with density $f$. From \cite{Kahle11}, we know that for both $\cP_n$ and $\cX_n$ the thermodynamic regime corresponds to the case $nr_n^d \to r \in (0,\infty)$, so that for such a radius regime we have that
	\[ \EXP{\beta_k(\C(\cP_n,r_n))} = \Theta(n), \qquad     \EXP{\beta_k(\C(\cX_n,r_n))} = \Theta(n). \]

	In proving limit results for  Betti numbers in these cases, much will depend on moment estimates for  the add-one cost function. The add-one cost function for a real-valued functional $F$ defined over finite point-sets $\cX$ is defined by
	\begin{equation}
	\label{eqn:add_One}
	D_xF(\cX)\  \definedas \ F(\cX \cup \{x\}) - F(\cX), \qquad x \in \mR^d.
	\end{equation}
	Our basic estimate follows.  For notational convenience, we write
	\[ \beta_k^n(\cX)\  \definedas \ \beta_k(\C(\cX,r_n)), \]
	where $\{r_n\}_{n \geq 1}$ is a sequence of radii to be determined.
	\begin{lemma}
		\label{lem:moments_add_One}
		Let $1\leq k \leq d-1$. For the Poisson point process $\cP_n$ and binomial point process $\cX_n$, with $nr_n^d \to r \in (0,\infty)$, we have that
		\beq
		\Delta_k  \ \definedas \  \max \left( \sup_{n \geq 1}\sup_{x \in \mR^d} \EXP{|D_x\beta^n_k(\cP_n)|^4}, \ \sup_{n \geq 1}\sup_{x \in \mR^d} \EXP{|D_x\beta^n_k(\cX_n)|^4}
		\right)
		\label{deltak}
		\eeq
		is finite
	\end{lemma}
	\begin{proof}
		The lemma is a consequence of the following simple bounds from Lemma \ref{lem:MV_complex_bounds}.
		\begin{eqnarray}
		|D_x\beta^n_k(\cP_n)| & \leq & \sum_{j=k}^{k+1}S_j(\cP_n,r_n ; \{x\})   \no \\
		&\leq& \left[\cP_n(B_x(r_n))\right]^k + \left[\cP_n(B_x(r_n))\right]^{k+1} \no \\
		& \leq &2\left[\cP_n(B_x(r_n))\right]^{k+1}, \no
		\end{eqnarray}
		and,  similarly,
		\beqq
		|D_x\beta^n_k(\cX_n)| &\leq& 2\left[\cX_n(B_x(r_n))\right]^{k+1} .
		\eeqq
		Set $r_*= \sup_{n \geq 1}\omega_dnr_n^d < \infty$, where $\omega_d$ is the volume of a $d$-dimensional unit ball. Let $\Poi(\lam)$ and $\Bin (n,p)$ denote the Poisson random variable with mean $\lam$ and the binomial random variable with parameters $n,p$ respectively. Then, we obtain that
		\begin{eqnarray}
		\EXP{|D_x\beta^n_k(\cP_n)|^4} & \leq & 16 \EXP{\left[\cP_n(B_x(r_n))\right]^{4(k+1)}}\  \leq \ 16 \EXP{\left[\Poi(r_*f^*)\right]^{4(k+1)}}, \no \\
		\EXP{|D_x\beta^n_k(\cX_n)|^4} & \leq & 16 \EXP{\left[\cX_n(B_x(r_n))\right]^{4(k+1)}} \ \leq \ 16 \EXP{\Big[\Bin (n,\frac{r_*f^*}{n})\Big]^{4(k+1)}}. \no
		\end{eqnarray}
		The lemma now follows from the boundedness of moments of Poisson and binomial random variables with constant means.
	\end{proof}

	\subsection{Strong laws}
	We begin with a  lemma giving  variance inequalities, which, en passant, establish weak laws for Betti numbers.
	
	\begin{lemma}
		\label{lem:weak_law_Poisson}
		For the Poisson point process $\cP_n$ and binomial point process $\cX_n$, with $nr_n^d \to r \in (0,\infty)$, and each $1\leq k \leq d-1$, there exists a positive constant $c_1$ such that for all $n \geq 1$,
		\beq
		\label{inequalities-rob}
		\VAR{\beta_k(\C(\cP_n,r_n))}  <  c_1 n, \qquad
		\VAR{\beta_k(\C(\cX_n,r_n))}  <  c_1 n.
		\eeq
		Thus, as $n \to \infty$,
		\beqq
		\label{weak-rob}
		n^{-1}\left[\beta_k(\C(\cP_n,r_n))\, -\, \EXP{\beta_k(\C(\cP_n,r_n))}\right]  \ \stackrel{\P}{\to}\  0,
		\eeqq
		and
		\beqq
		n^{-1}\left[\beta_k(\C(\cX_n,r_n)) \, -\,\EXP{\beta_k(\C(\cX_n,r_n))}\right]  \ \stackrel{\P}{\to}\  0.
		\eeqq
	\end{lemma}
	
	\begin{proof}
		Note firstly that the two weak laws \eqref{weak-rob}  follow  immediately from the upper bounds in
		\eqref{inequalities-rob} and Chebyshev's inequality.
		
		Thus it remains to prove \eqref{inequalities-rob}. The Poisson case is the easiest, since by Poincar\'{e}'s inequality (e.g.\ \cite[equation (1.8)]{Last11}), the Cauchy-Schwartz inequality and Lemma \ref{lem:moments_add_One},
		\beqq
		\label{eqn:variance_upper_bound_Poi}
		\VAR{\beta_k(\C(\cP_n,r)}  &\leq&
		\int_{\mR^d} \EXP{[D_x\beta^n_k(\cP_n)]^2} n f(x) \md x  \\
		&\leq& n \sqrt{\Delta_k},
		\eeqq
		where $\Delta_k<\infty$ is given by \eqref{deltak}.
		%
		%
		
		For the binomial case, we  need the Efron-Stein inequality (cf.\ \cite{Efron81} and for the case of random vectors, \cite[(2.1)]{Steele86}), which states that for a symmetric function
		$F\: (\mR^d)^n \to \mR$,
		\beqq
		\VAR{F(\cX_n)} &\leq& \smallhalf \sum_{i=1}^n\EXP{\big[F(\cX_n)-F(\cX_{n+1} \setminus \{X_i\})\big]^2} ,
		\eeqq
		where $\cX_n$ and $\cX_{n+1}$ are coupled so that $\cX_{n+1} = \cX_{n}\cup \{X_{n+1}\}$.
		Applying this inequality, we have
		\begin{eqnarray}
		\VAR{\beta_k(\C(\cX_n,r_n))}  &\leq& \smallhalf \sum_{i=1}^n\EXP{\left[\beta_k(\cX_n,r_n)-\beta_k(\cX_{n+1} \setminus \{X_i\},r_n)\right]^2} \no \\
		& = & \smallhalf \sum_{i=1}^n\E \Big\{\left[\beta_k(\cX_n,r_n)-\beta_k(\cX_n \setminus \{X_i\},r_n) \right.   \no  \\
		&&\qquad\qquad + \left. \beta_k(\cX_n \setminus \{X_i\},r_n) - \beta_k(\cX_{n+1} \setminus \{X_i\},r_n)\right]^2\Big\} \no \\
		& \leq & \smallhalf \sum_{i=1}^n 4\sqrt{\Delta_k} \no \\
		&=& 2n\sqrt{\Delta_k}, \label{eqn:variance_upper_bound_Bin}
		\end{eqnarray}
		where in the second inequality we have used Lemma \ref{lem:moments_add_One}. This completes the proof .
	\end{proof}
	Thanks to the recent bound of  \cite[Theorem 5.2]{Last14} (see Lemma \ref{lem:lower_bound_variance}), we can also give a lower bound for the Poisson point process in the case of $k = d-1$.
	\begin{lemma}
		\label{lem:var_lower}
		For the Poisson point process $\cP_n$ with $nr_n^d \to r \in (0,\infty)$, let $n_0$ be such that there is a set $A \subset \text{\rm supp}(f)$ with $A \oplus B_O(3r_n) \subset \text{\rm supp}(f)$ and $|A| > 0$ for all $n \geq n_0$. Then, there exists a positive constant $c_2$ such that, for all $n \geq n_0$ as above,
		\beq
		\label{inequalities-rob1}
		\VAR{\beta_{d-1}(\C(\cP_n,r_n))}  >  c_2 n.
		\eeq
	\end{lemma}
	\begin{remark}
		\label{rem:duality}
		Note that from the universal coefficient theorem (\cite[Theorem 45.8]{Munkres84}) and Alexander duality (\cite[Theorem 16]{Spanier66}), we have that\footnote{The $\tilde{H}_k$ are the reduced homology groups and it suffices to note that $\tilde{H}_k \cong H_k$ for $k \neq 0$ and $H_0 \cong \tilde{H}_0 \oplus \mathbb{F}$.}
		\beqq
		\tilde{H}_k(\cCB(\cP_n,r)) \ \cong \  \tilde{H}_{d-k-1}(\mR^d \setminus \cCB(\cP_n,r)).
		\eeqq
		Thus
		\beqq
		\beta_{d-1}(\cCB(\cP_n,r)) \ =\  \beta_0(\mR^d \setminus \cCB(\cP_n,r)) - 1.
		\eeqq
		$\beta_0(\mR^d \setminus \cCB(\cP_n,r))$ is nothing but the number of components of the vacant region of the Boolean model, which is easier to analyse and this will play a crucial role in our proof.
	\end{remark}
	\begin{proof}
		The proof will be based on Lemma \ref{lem:lower_bound_variance} and the duality argument of Remark \ref{rem:duality}. The finiteness of moments required by this lemma is guaranteed by Lemma \ref{lem:moments_add_One}. Choose $n \geq n_0$ for $n_0$ as defined in the statement of the lemma and also the set $A$ guaranteed by this assumption. Let $x \in A$. So we now  have to show that, for each $1 \leq k\leq d-1$, there exists an  $m$ (depending on $k$ and $d$ only) and a finite set of points $\{z_1,\ldots,z_m\} \in B_O(2r_n)$  such that for some constants $c, c_* \in (0,1)$ and for all $(y_1,\ldots,y_m) \in \prod_{i=1}^mB_{x+z_i}(c_*r_n)$,
		\beq
		\label{eqn:lower_bound_add_Onea}
		\pr{D_x\beta_k^n(\cP_n \cup \{y_1,\ldots,y_m \}) \leq -1} > c,
		\eeq
		and
		\beq
		\label{eqn:lower_bound_add_Oneb}
		D_x\beta_{d-1}^n(\cP_n \cup \{y_1,\ldots,y_m\}) \leq 0,
		\eeq
		with probability one. Though not explicitly mentioned, it is to be understood that the above choices of $m,z_i,c,c_*$ are not dependent on $x \in A$.
		The above two inequalities imply that
		\[ \left| \EXP{ D_x\beta_{d-1}^n(\cP_n \cup \{y_1,\ldots,y_m\})} \right| \  \geq \ c,\]
		so that condition \eqref{condn:variance_lower_bound} required in Lemma \ref{lem:lower_bound_variance} is satisfied for the add-one cost function with the constant $c$ above, and the lower  bound to the variance given there holds. Though, in this proof we require the construction only for $k = d-1$, we have stated one of the inequalities for all $k$ as this will be important to the variance lower bound argument in Theorem \ref{thm:clt_Poisson_Binomial}.
		
		Moreover, it is easy to check that, given our choice of $\{z_1,\ldots,z_m\}$ and  $c_*r_n$ (in place of $r$ in Lemma \ref{lem:lower_bound_variance}), the  bound in  \eqref{condn:variance_lower_bound-A} can be further bounded from below by $c_2 n$ for some $c_2>0$ (depending only on $f$, $A$, $k$, $r$ and $d$). This will prove the lemma.
		
		Thus, all that remains is to find an  $m$ and construct $z_1,\ldots,z_m$ satisfying the above conditions.
		
		Fix $k \in \{1,\ldots,d-1\}$. Let $S^k$ denote the unit $k$-dimensional sphere, and  embed it via the usual inclusion in the unit sphere in $\mR^d$.   For $\ep < \smallquarter$ let
		$S_\ep^k=\{x\in \mR^d\: \min_{y\in S^k} \|x-y\|\leq \ep\}$ denote the $\ep$-thickening of $S^k$.

		Now  choose $m$ large enough (but depending only on $k$ and $d$ only) such that there exist points $v_1,\ldots,v_{m}$ in $\mR^d$ so that
		\beq
		\label{mkballs}
		S_\ep^k  \ \subset\  \bigcup_{i =1}^{m}B_{v_i}(1)\  \subset \ \left(B_O(\smallquarter)\right)^c \eeq
		and, for all $0\leq j\leq d-1$,
		\beq
		\label{betak-rob}
		\beta_j\left(\C(\{v_1,\ldots,v_{m}\},1)\right) \ =\   \beta_j(S^k)=\beta_j(S^k_\ep).
		\eeq
		(Recall that $ \beta_j(S^k)=0$ for $j\neq 0,k$, while $\beta_0(S^k) = \beta_k(S^k) =1$.)
		
		Now, if needed choose $m$ larger such that there is a $c_* > 0$ for which all $(y_1,\ldots,y_m) \in \prod_{i=1}^mB_{c_*}(v_i)$ satisfy \eqref{mkballs} and \eqref{betak-rob}. Note that by scaling we have, for all
		$\{y_1,\ldots,y_m\} \in \prod_{i=1}^mB_{r_nv_i}(c_*r_n) $,
		\beqq
		r_nS^k_{r_n\ep}\ \subset \ \bigcup_{i =1}^{m}B_{y_i}(r_n) \ \subset\  \left(B_O(r_n/{4})\right)^c,
		\eeqq
		while $\beta_k(\C(\{y_1,\ldots,y_m\},r_n)) = 1$.
		
		Setting $z_i = r_nv_i$ for $i = 1,\ldots, m$, we have that $z_i \in B_O(2r_n)$ as required as well as $c_*$ chosen as above ensures the size requirements we need.  So, what remains is to show that \eqref{eqn:lower_bound_add_Onea} and \eqref{eqn:lower_bound_add_Oneb} hold for $\{y_1,\ldots,y_m\}\in B_{n,m} := \prod_{i=1}^mB_{x+z_i}(c_*r_n)$.
		
		%
		On the other hand, the structure of $B_{n,m}$ implies that, for $\{y_1,\ldots,y_m\}\in B_{n,m}$,
		\beqq
		\pr{D_x\beta_k^n(\cP_n \cup \{y_1,\ldots,y_m \}) \leq -1 \ \big|\  \cP_n(B_x(2r_n)) = 0} \ =\  1.
		\eeqq
		Furthermore, it is immediate from Poisson void probabilities that
		\beqq
		\pr{\cP_n(B_x(2r_n)) = 0} \ \geq\    e^{-f^*n\omega_dr_n^d} \ \geq\  e^{-r_*f^*} > 0,
		\eeqq
		where $r_*:= \sup_{n \geq 1}n\omega_dr_n^d$. These two facts together imply that
		(\ref{eqn:lower_bound_add_Onea}) holds with $c = e^{-r_*f^*} > 0$.
		
		We now turn to the second of these inequalities (which is only for $k = d-1$), for which we need the nerve theorem (\cite[Theorem 10.7]{Bjorner95}) along with duality argument (Remark \ref{rem:duality}). The nerve theorem allows us to prove the inequality for $\beta_{d-1}(\cCB(\cP_n,r_n))$ instead of $\beta_{d-1}(\cC(\cP_n,r_n))$, and the duality argument further reduces our task to proving
		\beq
		\label{eqn:vac_comp_bd}
		D_x\beta_0^n(\mR^d \setminus \cCB(\cP_n \cup \{y_1,\ldots,y_m\},r_n)) \leq 0,
		\eeq
		with probability one.
		
		Set $V_n := \mR^d \setminus \cCB(\cP_n \cup \{y_1,\ldots,y_m\},r_n)$. Since $x \oplus r_nS_{r_n\ep}^{d-1}  \ \subset\  \bigcup_{i =1}^{m}B_{y_i}(r_n)$, we have that $V_n$ is the disjoint union of $V_n \cap B_x(r_n)$ and $V_n \cap B_x(r_n)^c$. Thus,
		\[ \beta_0^n(\mR^d \setminus \cCB(\cP_n \cup \{y_1,\ldots,y_m\},r_n)) = \beta_0(V_n \cap B_x(r_n)) + \beta_0(V_n \cap B_x(r_n)^c).\]
		So,
		\begin{eqnarray*}
			&  & D_x\beta_0^n(\mR^d \setminus \cCB(\cP_n \cup \{y_1,\ldots,y_m\},r_n)) \\
			& = & \beta_0^n(\mR^d \setminus \cCB(\cP_n \cup \{x,y_1,\ldots,y_m\},r_n)) - \beta_0(V_n \cap B_x(r_n)) - \beta_0(V_n \cap B_x(r_n)^c) \\
			& = &  \beta_0(V_n \cap B_x(r_n)^c) - \beta_0(V_n \cap B_x(r_n)) - \beta_0(V_n \cap B_x(r_n)^c) \\
			& = & - \beta_0(V_n \cap B_x(r_n)) \leq 0,
		\end{eqnarray*}
		where in the second equality, we have used the fact
		\[ B_x(r_n) \subset \cCB(\cP_n \cup \{x,y_1,\ldots,y_m\},r_n). \]
		This proves (\ref{eqn:vac_comp_bd}) and hence  we have (\ref{eqn:lower_bound_add_Oneb}), which was all that was required to complete the proof.
	\end{proof}
	
	Our next main result is a concentration inequality for $\beta_k  (\C(\cX_n,r_n)$.
	
	\begin{theorem}
		\label{thm:conc_ineq_Binomial}
		Let $1\leq k \leq d-1$,  $\cX_n$ be a binomial point process, and assume that $nr_n^d \to r\in (0,\infty)$. Then, for any $a > \frac{1}{2}$ and $\ep > 0$, for $n$ large enough,
		\[ \pr{\big|\beta_k(\C(\cX_n,r_n) - \EXP{\beta_k(\C(\cX_n,r_n)}\big| \geq \ep n^a} \ \leq \ \frac{{C}}{\ep}n^{2k+2-a}\exp(-n^{\gamma}),\]
		where $\gamma= {(2a-1)}/{4k}$ and ${C} > 0$ is a constant depending only on $a,r,k,d$ and the density $f$.
	\end{theorem}
	\remove{ \begin{remark}
			In particular, choose $\tilde{a} < \frac{1}{2}$ and set $a = 1 - \tilde{a}$. Then, with constants $\gamma, \widehat{c}, \tilde{c}$ as in Theorem \ref{thm:conc_ineq_Binomial}, we obtain that
			\[ \pr{|\frac{\beta_k(\C(\cX_n,r_n) - \EXP{\beta_k(\C(\cX_n,r_n)}}{n}| \geq \epsilon n^{-\tilde{a}}} \leq \frac{\widehat{c}}{\ep}n^{2(k+1)-a}\exp(-n^{\gamma}).\]
		\end{remark}
	}
	The proof, close to that of \cite[Theorem 3.17]{Penrose03},   is based on  a concentration inequality for martingale differences. \\

	\begin{proof}
		Fix $n \in \mN$. Let $Q_{n,i}$ be a partition of $\mR^d$ into cubes of side length $r_n$. Define the set $\mathbb{A}_n$ as follows:
		\[ \mathbb{A}_n\ \definedas \ \left\{ \cX\: |\cX| = n,\ \forall i,\ \cX(Q_{n,i} \cap \text{supp}(f)) \leq \max(r,1)n^{\gamma} \right\}. \]

		For large enough $n$, since $\cX_n(Q_{n,i})$ is stochastically dominated by a $\Bin (n,f^*r_n^d)$ random variable, elementary bounds (e.g.\
		\cite[Lemma 1.1]{Penrose03}), yield that
		\[ \pr{\cX_n \notin \mathbb{A}_n}\  \leq \ {c}_1 n \exp(-n^{\gamma}), \]
		for some constant ${c_1}$. Since the above bound is dependent only on the mean of the binomial random variable and $r$, the constants $d$ and $r$ are suppressed. Recall that the points of $\cX_n$ are denoted by $X_1,\dots,X_n$, and let $\cF_i = \sigma(X_1,\ldots,X_i)$ be the sequence of $\sigma$-fields they generate, with $\cF_0$ denoting the trivial $\sigma$-field. (The actual ordering of the $X_j$ will not be important.) We can define the finite martingale
		\beqq
		M^{(n)}_{i}\ \definedas\  \EXP{\beta_k(\C(\cX_n,r_n))\mid\cF_i},
		\eeqq
		for $0=1,\dots,n$,  along with the corresponding  martingale differences
		\[ D^{(n)}_{i}\ \definedas \ \EXP{\beta_k(\C(\cX_n,r_n))\mid \cF_i} - \EXP{\beta_k(\C(\cX_n,r_n)) \mid \cF_{i-1}}, \]
		$i=1,\dots,n$, and $D^{(n)}_0 =0$.
		Writing
		\beqq
		\beta_k(\C(\cX_n,r_n)) - \EXP{\beta_k(\C(\cX_n,r_n))} \ =\  \sum_{i=0}^nD_{i}^{(n)},
		\eeqq
		and  setting $\cX^{i}_n = \cX_{n+1}\setminus \{X_i\}$, we have that $D_{i}^{(n)}$ can be represented as
		\[ D_{i}^{(n)} \ = \  \EXP{\beta_k(\C(\cX_n,r_n)) - \beta_k(\C(\cX^i_n,r_n)) \mid \cF_i}. \]

		Let $A_{i,n}:= \{ \cX_n \in \mathbb{A}_n , \cX^i_n \in \mathbb{A}_n \}$. Then,  recalling that
		$S_j(\cX,r;A)$   denotes the number of $j$-simplices with at least one vertex in $A$, and appealing again to Lemma \ref{lem:MV_complex_bounds}, we have that, conditioned on the event $A_{i,n}$, for $n$ large enough,
		\[ |\beta_k(\C(\cX_n,r_n)) - \beta_k(\C(\cX^i_n,r_n))| \  \leq \  \sum_{j=k}^{k+1}S_j(\cX_{n+1},r_n ; \{X_i,X_{n+1}\}) \ \leq\  {c}_2(rn^{\gamma})^k. \]
		In all cases, we have the universal bound $|\beta_k(\C(\cX_n,r_n)) - \beta_k(\C(\cX^i_n,r_n))| \leq n^k$, and so,
		\begin{eqnarray*}
			\big|D_{i}^{(n)}\big| & = & \EXP{\big(\big|D_{i}^{(n)}\big|  \1_{A_{i,n}}\big) \big| \cF_i} + \EXP{\big(\big|D_{i}^{(n)}\big| \1_{A_{i,n}^c}\big) \big| \cF_i} \\
			& \leq & {c}_2(rn^{\gamma})^k + n^{k}\pr{A^c_{i,n}|\cF_i}.
		\end{eqnarray*}
		Defining $B_{i,n}:= \{ \pr{A^c_{i,n}\mid\cF_i} \leq n^{-k} \}$,  Markov's inequality implies
		\[\pr{B_{i,n}^c} \ \leq \ n^{k}\EXP{\pr{A^c_{i,n}\mid\cF_i}} \ =\  n^{k}\pr{A_{i,n}^c} \ \leq\  2 {c}_1n^{k+1} \exp(-n^{\gamma}). \]
		Thus, since $|D_{i}^{(n)}|\1_{B_{i,n}} \leq {c}_3(rn^{\gamma})^k$, using
		\cite[Lemma 1]{Chalker99},
		we have that for any $b_1,b_2 > 0$,
		\[\pr{\Big|\sum_{i = 1}^n D_{i}^{(n)}\Big| > b_1} \  \leq\  2\exp(-\frac{b_1^2}{32nb_2^2}) + \left(1 + \frac{2\sup_i \|D_{i}^{(n)}\|_{\infty}}{b_1}\right)\sum_{i=1}^n  \pr{\left|D_i\right| > b_2}. \]
		Choosing $b_1 = \ep n^a$ and $b_2 = {c}_3(rn^{\gamma})^k$, we have
		\begin{eqnarray*}
			&& \pr{|\beta_k(\C(\cX_n,r_n)) - \EXP{\beta_k(\C(\cX_n,r_n))}| \geq \ep n^a} \\
			&& \quad\quad \leq \  2\exp(-\frac{\ep^2 n^{2a}}{32n{c}_3^2(rn^{\gamma})^{2k}}) \ + \ \left( 1 +2\ep^{-1}n^{k-a}\right)\sum_{i=1}^n\pr{|D_{i}^{(n)}| > {c}_3(rn^{\gamma})^k} \\
			&& \quad\quad \leq \ 2\exp(-\frac{\ep^2 n^{2a}}{32n{c}_3^2(rn^{\gamma})^{2k}}) \ +\ \left( 1 +2\ep^{-1}n^{k-a}\right)\sum_{i=1}^n\pr{B_{i,n}^c} \\
			&& \quad\quad \leq \ 2\exp(-\frac{\ep^2 n^{2a-2k\gamma -1}}{32{c}_3^2r^{2k}}) \ +\ \left( 1 + 2\ep^{-1}n^{k-a}\right)2{c}_1n^{k+2}\exp(-n^{\gamma}) \\
			&& \quad\quad \leq \ \frac{{c}}{\ep}n^{2k+2-a}\exp(-n^{\gamma}),
		\end{eqnarray*}
		{for large enough $n$ and} a constant ${c} > 0$ which depends only on $a,r,k,d$. The last inequality above follows from the fact that $\gamma < a - \frac{1}{2}$ and hence, for large enough $n$, $\exp(-n^{\gamma})$ is the dominating term in the penultimate expression.
	\end{proof}

	We now finally have the ingredients needed to lift the weak laws of Lemma \ref{lem:weak_law_Poisson} to the promised strong convergence.

	\begin{theorem}
		\label{thm:strong_law_Poisson}
		For the Poisson point process $\cP_n$ and binomial point process $\cX_n$, with $nr_n^d \to r \in (0,\infty)$, and each $1\leq k \leq d-1$,
		we have, with probability one,
		\[ \lim_{n \to \infty} n^{-1} \left[\beta_k(\C(\cP_n,r_n)) -   \EXP{\beta_k(\C(\cP_n,r_n)} \right]\ =\ 0, \]
		and
		\[ \lim_{n \to \infty}   n^{-1}\left[ \beta_k(\C(\cX_n,r_n)) -   \EXP{\beta_k(\C(\cX_n,r_n)} \right]\ =\  0. \]
		%
		
		%
	\end{theorem}
	
	\begin{proof}
		By choosing $a = 1$ in Theorem \ref{thm:conc_ineq_Binomial} and summing over $n$, we have, for all $\ep > 0$,
		\[ \sum_{n \geq 1} \pr{n^{-1}{\left|\beta_k(\C(\cX_n,r_n) - \EXP{\beta_k(\C(\cX_n,r_n)}\right|}       \geq \ep} \ <\  \infty.\]
		The Borel-Cantelli lemma immediately implies the strong law for $\beta_k(\C(\cX_n,r_n))$.
		
		Turning to the Poisson case, we shall use the standard coupling of $\cP_n$ to $\cX_n$ to complete the proof of the theorem. Let $N_n$ be a Poisson random variable with mean $n$. Then, by choosing $\cP_n = \{X_1,\ldots,X_{N_n}\}$, we have coupled it with $\cX_n = \{X_1,\ldots,X_n\}$, $n\geq 1$.  By Lemma \ref{lem:MV_complex_bounds}, we have that
		\beqq
		&&|\beta_k(\C(\cP_n,r_n) - \beta_k(\C(\cX_n,r_n)|
		\\ &&\qquad\qquad
		\leq\ S_k(\cP_n \cup \cX_n,r_n ; \cP_n \bigtriangleup \cX_n) \ +\  S_{k+1}(\cP_n \cup \cX_n,r_n ; \cP_n \bigtriangleup \cX_n) .
		\eeqq

		Now note that,
		\[ S_k(\cP_n \bigcup \cX_n,r_n ; \cP_n \bigtriangleup \cX_n) \ =\  |S_k(\cP_n ,r_n) - S_k(\cX_n ,r_n)|, \]
		with a similar inequality holding for $S_{k+1}()$. From \cite[Theorem 2.2]{Penrose07a} and the remarks following that result, we have that there exists a constant $\widehat{S}_k(f) \in (0,\infty)$ such that, with probability one,
		\[ \lim_{n \to \infty} \frac{S_k(\cX_n ,r_n)}{n} \ =\
		\lim_{n \to \infty} \frac{S_k(\cP_n ,r_n)}{n} \ =\  \widehat{S}_k(f) . \]
		A similar limit law holds true for $S_{k+1}$,  and so we have that, with probability one,
		\[ \lim_{n \to \infty}\frac{S_k(\cP_n \cup \cX_n,r_n ; \cP_n \bigtriangleup \cX_n)}{n} \ =\  \lim_{n \to \infty}\frac{S_{k+1}(\cP_n \cup \cX_n,r_n ; \cP_n \bigtriangleup \cX_n)}{n}
		\ =\  0 .\]
		Thus,
		\[ \frac{|\beta_k(\C(\cP_n,r_n) - \beta_k(\C(\cX_n,r_n)|}{n} \ \stackrel{n \to \infty}{\rar}\  0, \]
		and  the strong law for $\beta_k(\C(\cP_n,r_n)$  follows.
	\end{proof}

	
	\subsection{Central Limit Theorem}
	\label{sec:clt}
	
	We have finally come to main result of this section: central limit theorems for Betti numbers. \remove{In doing so, we shall restrict ourselves somewhat further, and assume that the underlying density for the Possion and binomial processes is uniform on the unit cube in $\mR^d$.
		In this section we shall   slightly modify the basic setup,  so as to more readily apply general techniques on central limit theorems for geometric functionals of Poisson and binomial point processes.
		(cf.\  Appendix, Theorem \ref{thm:clt_Penrose01}, a version of
		\cite[Theorem 2.1]{Penrose01}.) }
	
	We start with some definitions  from percolation theory for the Boolean model on Poisson processes (\cite{Meester96}) needed for the proof of the Poisson central limit theorem.  Recall firstly that we say that a subset $A$ of $\mR^d$ \emph{percolates} if it contains  an unbounded connected component of $A$.
	
	Now let $\cP$ be a stationary Poisson point process on $\mR^d$ with unit intensity.
	(Unit intensity is for notational convenience only. The arguments of this section will work for any constant intensity.)
	\remove{Recall that the Boolean model was defined in Definition (\ref{defn:boolean}) and that the relevance of the Boolean model to the study of \Cech complexes stems from the nerve theorem (\cite[Theorem 10.7]{Bjorner95}) which implies that  $\cCB(\cX,r)$ and $\C(\cX,r)$ are homotopy equivalent for any finite set $\cX$.}
	We define the critical (percolation)  radii for $\cP$ as follows:
	\[ r_c(\cP)\ \definedas \ \inf \{ r\: \pr{\mbox{$C(\cP,r)$ percolates}} > 0 \}, \]
	and,
	\[ r^*_c(\cP)\ \definedas \ \sup \{ r\: \pr{\mbox{$\mR^d \setminus C(\cP,r)$ percolates}} > 0 \} .\]
	By Kolmogorov's zero-one law, it is easy to see that the both of the probabilities inside the {infimum and supremum}  here are  either $0$ or $1$. The first critical radius is called the \emph{critical radius for percolation of the occupied component} and the second is the \emph{critical radius for percolation of the vacant component}.
	
	We define the \emph{interval of co-existence},  $I_d(\cP)$, for which unbounded components of both the Boolean model and its complement  co-exist, as follows:
	\beqq
	I_d(\cP) \  =\  \begin{cases}   (r_c,r_c^*]
		&\text{if $\pr{\mbox{$C(\cP,r_c)$ percolates}} = 0$,}\\
		[r_c,r_c^*]   &\text{otherwise}.
	\end{cases}
	\eeqq
	From \cite[Theorem 4.4 and Theorem 4.5]{Meester96}, we know that $I_2(\cP)= \emptyset$ and from \cite[Theorem 1]{Sarkar97} we know that $I_d(\cP) \neq \emptyset$ for $d \geq 3$. In high dimensions, it is known that $r_c \notin I_d(\cP)$ (cf.\ \cite{Tanemura96}).
	
	We now need a little additional notation.  Let $\{B_n\}_{n \geq 1}$ be a sequence of bounded Borel subsets in $\mR^d$ satisfying the following four conditions:
	\begin{enumerate}
		\item[(A)] $|B_n| = {n}$, for all $n \geq 1$.
		\item[(B)] $\bigcup_{n \geq 1}\bigcap_{m \geq n} B_m \ =\  \mR^d$.
		\item[(C)] $ {|(\partial B_n)^{(r)}|}/{n} \to 0$, for all $r > 0$.
		\item[(D)] There exists a constant $b_1$ such that diam$(B_n) \leq b_1n^{b_1}$, where
		diam$(B)$ is the diameter of $B$.
	\end{enumerate}

	In a moment we shall state and  prove a central limit theorem for the sequences of the form $\beta_k(\C(\cP\cap B_n,r))$, when the $B_n$ are as above. Setting up the central limit theorem for the binomial case requires a little more notation.

	In particular, we write $\cU_n$ to denote  the point process obtained by choosing $n$ points uniformly in $B_n$, and  call this
	the {\it  extended binomial point process}.   This is a natural binomial counterpart to the Poisson point process $\cP \cap B_n$.
	
	We finally have all that we need to formulate the main central limit theorem.

	\begin{theorem}
		\label{thm:clt_Poisson_Binomial}
		Let $\{B_n\}$ be a sequence of sets in $\mR^d$ satisfying conditions (A)--(D) above, and let
		$\cP$ and $\cU_n, \ n \geq 1$, respectively,  be the unit intensity Poisson process and the extended binomial point process described above.  Take $k \in \{1,\ldots,d-1\}$ and $r \in (0,\infty)$. Then there exists a constant $\sigma^2 > 0$ such that, as $n \to \infty$,
		\[ n^{-1}\VAR{\beta_k(\C(\cP\cap B_n,r)} \ \to \ \sigma^2, \]
		and
		\[ n^{-1/2}\left(\beta_k(\C(\cP\cap B_n,r)) - \EXP{\beta_k(\C(\cP\cap B_n,r))}\right)\  \Rar\  N(0,\sigma^2). \]
		Furthermore, for $r \notin I_d(\cP)$, there exists a $\tau^2$ with $0 < \tau^2 \leq \sigma^2$ such that
		\[ n^{-1}\VAR{\beta_k(\C(\cU_n,r)} \to \tau^2 , \]
		and
		\[ n^{-1/2}\left(\beta_k(\C(\cU_n,r)) - \EXP{\beta_k(\C(\cU_n,r))}\right) \ \Rar\  N(0,\tau^2). \]
		The constants $\sigma^2$ and $\tau^2$ are independent of the sequence $\{B_n\}$.
	\end{theorem}
	\begin{remark}
		\label{rem:top_homology}
		The condition  $r \notin I_d(\cP)$, needed for the binomial central limit theorem, is rather irritating, and we are not sure if it is necessary or an artefact of the proof. It is definitely not needed for the case
		$k=d-1$. To see this, note that from the duality argument of Remark \ref{rem:duality}, we have that
		\beqq
		\beta_{d-1}(\C(\cP \cap B_n,r)) \ =\  \beta_0(\mR^d \setminus \C(\cP\cap B_n,r)) - 1.
		\eeqq
		However, $\mR^d \setminus \C(\cP\cap B_n,r)$ is nothing but the vacant component of the Boolean model, and central limit theorems for {$\beta_0(\mR^d \setminus \C(\cX\cap B_n,r))$} for both Poisson and binomial point processes are given in \cite[p1040]{Penrose01}  for all $r \in (0,\infty)$. By the above duality arguments, this proves both the central limit theorems  of Theorem \ref{thm:clt_Poisson_Binomial}, when $k=d-1$,   and without the requirement that $r \notin I_d(\cP)$.
	\end{remark}
	\begin{proof}
		Since the theorem is somewhat of an omnibus collection of results, the proof is rather long. Thus we shall break it up into signposted segments.
		
		\vskip0.1truein
		\noindent{\it I. Poisson central limit theorem:}\ \
		Since $\beta_k(\C(\,\cdot\, ,r))$ is a translation invariant functional over finite subsets of $\mR^d$, we  need only check that {Conditions 3} and 4 in Theorem \ref{thm:clt_Penrose01}, along with the weak stabilization of \eqref{weakstab-rob},  hold in order to prove the convergence of variances and the asymptotic normality in the Poisson case.
		The strict positivity of $\sigma^2$ will follow from that of $\tau^2$, to be proven below.
		
		We treat each of the three necessary conditions separately.
		
		\vskip0.1truein
		\noindent{\it (i) Weak stabilization:} Firstly, we shall show that there exist a.s.\ finite random variables $D_{\beta_k}(\infty),$ and $R$ such that, for all $\rho > R$,
		\begin{equation}
		\label{eqn:weak_stabilizing_spheres}
		(D_O\beta_k^r)(\cP \cap B_O(\rho)) \ =\  D_{\beta_k}(\infty),
		\end{equation}
		where  $\beta_k^r(\cX)\definedas\beta_k(\C(\cX,r))$ for any finite point-set $\cX$. Then, we shall complete the proof of weak stabilization by showing the above for any $\mathfrak{B}$-valued sequence of sets $A_n$ tending to $\mR^d$. (See the paragraphs preceding Theorem  \ref{thm:clt_Penrose01} for the definition of $\mathfrak{B}$.)

		For any $\rho>2r$, define the simplicial complexes
		\beqq
		\cK_{\rho} &=& \C((\cP \cap B_O(\rho)) \cup \{0\},r), \\
		\cK_{\rho}' &=& \C(\cP \cap B_O(\rho),r),\\
		\cK'' &=& \C((\cP \cap B_O(2r)) \cup \{0\},r),\\
		\cL &=& \cK_{\rho}' \cap \cK'',
		\eeqq
		and note that $\cK_{\rho}=\cK_{\rho}'\cup\cK''$ and that, as implied by the notation, $\cL$ and $\cK''$ do not  depend on $\rho$.
		
		From the second part of Lemma \ref{lem:MV_complexes}, we have that
		\beq
		\no
		(D_O\beta_k^r)(\cP \cap B_O(\rho)) &=& \beta_k(\cK_{\rho})-\beta_k(\cK_{\rho}')
		\\ &=& \beta_k(\cK'')+ \beta(N_k^{\rho}) + \beta(N_{k-1}^{\rho})-\beta_k(\cL),
		\label{eqn:MV_spheres}\eeq
		where  $N_k^{\rho}$ is the kernel of the induced homomorphism
		\[ \lam_k^{\rho}\: H_k(\cL)\  \to\ H_k(\cK_{\rho}') \oplus H_k(\cK'')\]
		Hence, all that remains is to show that $\beta(N_j^{\rho})$, $j=k,k-1$, remain  unchanged as $\rho$
		increases beyond some random variable $R$. Since these variables  are integer valued, it suffices to show that they are increasing and bounded to prove (\ref{eqn:weak_stabilizing_spheres}). We shall do this  for $\beta(N_k^{\rho})$. The same proof also works for $\beta(N_{k-1}^{\rho})$.
		
		The boundedness is immediate,  since
		\beqq
		\beta(N_k^{\rho})\ \leq \  \beta_k(\cL)\  \leq\  \Phi(B_O(2r))^{k+1}\  < \ \infty, \quad \text{a.s.}
		\eeqq
		All that remains to show is that $\beta(N_k^{\rho})$ is increasing. Let $\rho_1,\rho_2$ be such that $2r < \rho_1 \leq \rho_2$. We need to show that $\beta(N_k^{\rho_1}) \leq \beta(N_k^{\rho_2})$.
		
		Since $\cL \subset \cK_{\rho_1} \subset \cK_{\rho_2}$, we have the  corresponding simplicial maps defined by the respective inclusions (see Section \ref{sec:topology}) and hence the following homomorphisms:
		\[  H_k(\cL) \ \stackrel{\lam_k^{\rho_1}}{\to}\  H_k(\cK_{\rho_1}') \oplus H_k(\cK'')
		\  \stackrel{\eta}{\to}\  H_k(\cK_{\rho_2}') \oplus H_k(\cK''). \]
		Also, by the functoriality of homology, $\lam_k^{\rho_2} = \eta \circ \lam_k^{\rho_1}$. Since $\ker \, \eta \subset \ker \, \eta' \circ \eta$ for any two homomorphisms $\eta,\eta'$, we have that
		\begin{equation}
		\label{eqn:increasing_kernel}
		\beta(N_k^{\rho_1}) \ =\  \beta(\ker \, \lam_k^{\rho_1})\  \leq\  \beta(\ker \, \eta \circ \lam_k^{\rho_1}) \  =\   \beta(N_k^{\rho_2}).
		\end{equation}
		This proves that $\beta(N_k^{\rho})$ is increasing in $\rho$, as, similarly, is $\beta(N_{k-1}^{\rho})$. Combining the convergence of $\beta(N_j^{\rho})$, $j=k,k-1$, with (\ref{eqn:MV_spheres}) gives (\ref{eqn:weak_stabilizing_spheres}).
		
		Now, let $A_n$ be a $\mathfrak{B}$-valued sequence of sets tending to $\mR^d$. To complete the proof of weak stabilization, we need to show that there exists an integer valued random variable $N$ such that for all $n > N$,
		\begin{equation}
		\label{eqn:weak_stabilizing}
		(D_O\beta_k^r)(\cP \cap A_n)\ =\  D_{\beta_k}(\infty),
		\end{equation}
		where, as before, $\beta_k^r(\cX)\definedas\beta_k(\C(\cX,r))$ for any finite point-set $\cX$. Firstly, choose $R$ as in (\ref{eqn:weak_stabilizing_spheres}) and WLOG assume $R>2r$. In particular, this implies that $\beta(N_j^{\rho})$ remains constant for $\rho > R$ for $j = k-1,k$. Since $\cP \cap B_O(R + 1)$ is a.s.\ finite and $\bigcup_{n \geq 1}\bigcap_{m \geq n}A_m = \mR^d$, there exists an a.s.\  finite random variable $N^*$ such that
		\[\{0\} \cup \cP \cap B_O(R + 1) \subset \bigcap_{m \ \geq\   N^*} A_m.\]
		Hence, $\{0\} \cup  \cP \cap B_O(R + 1) \subset A_n$ for all $n > N^*$. Let $n > N^*$,  and note that since $A_n\in \mathfrak{B}$, diam$(A_n)<\infty$ and so we can choose $R_n<\infty$ such that $A_n \subset B_O(R_n)$. Define the simplicial complexes
		\beqq
		\cK &=& \C((\cP \cap B_O(R + 1)) \cup \{0\},r), \\
		\cK_n &=& \C((\cP \cap A_n) \cup \{0\},r), \\
		\cK^*_n &=& \C((\cP \cap B_O(R_n)) \cup \{0\},r), \\
		\cK_n' &=& \C(\cP \cap A_n,r),\\
		\cK'' &=& \C((\cP \cap B_O(2r)) \cup \{0\},r),\\
		\cL &=& \cK_n' \cap \cK'',
		\eeqq
		where, again, $\cL$ and $\cK''$ do not depend of $n$. Now applying the second part of Lemma \ref{lem:MV_complexes}, since $\cK_n = \cK_n' \cup \cK''$,we have that
		\beq
		\no
		(D_O\beta_k^r)(\cP \cap A_n) &=& \beta_k(\cK_n)-\beta_k(\cK_n')\\
		&=& \beta_k(\cK'')+ \beta(M_k^n) + \beta(M_{k-1}^n)-\beta_k(\cL),
		\label{eqn:MV_sets}
		\eeq
		where $M_j^{n},j=k,k-1$, is the kernel of the induced homomorphism
		\[\gamma_j^n\:  H_j(\cL) \ \to\  H_j(\cK_n') \oplus H_j(\cK'').\]
		Again, to prove (\ref{eqn:weak_stabilizing}), all we need to show is that
		$\beta(M_j^n)$, $j=k,k-1$, remain constant for any $n > N^*$.
		
		To see this, start by noting that, by the choice of $n,R,R_n$,  we have the following inclusions:
		\beqq
		\cL \subset \cK \subset \cK'_n \subset \cK^*_n.
		\eeqq
		Hence the corresponding simplicial maps give rise to the following induced homomorphisms:
		\[ H_k(\cL) \ \stackrel{\eta_1}{\to}\  H_k(\cK) \oplus H_k(\cK'') \ \stackrel{\eta_2}{\to}\  H_k(\cK'_n) \oplus H_k(\cK'')\  \stackrel{\eta_3}{\to}\ H_k(\cK^*_n) \oplus H_k(\cK''). \]
		Note that $\gamma_k^n = \eta_2 \circ \eta_1$. Also, from the choice of $R, \cK$ and $\cK_n^*$, we have that
		\[ \beta(N_k^{R+1}) \ =\  \beta(\ker \, \eta_1)\  =\  \beta(\ker \, \eta_3 \circ \eta_2 \circ \eta_1),\]
		where $N_k^\rho$ for any $\rho \geq 2r$ was defined after (\ref{eqn:MV_spheres}). Now, by an argument similar to that used to obtain (\ref{eqn:increasing_kernel}), we have the following inequality:
		\[ \beta(\ker \, \eta_1)\  \leq\  \beta(M_k^n) \ =\  \beta(\ker \, \eta_2 \circ \eta_1)\  \leq \ \beta(\ker \, \eta_3 \circ \eta_2 \circ \eta_1).\]
		Thus, we have that $\beta(M_k^n) = \beta(N_k^{R+1})$ for $n > N^*$ with a corresponding result holding for
		$\beta(M_{k-1}^n)$.   Using this in (\ref{eqn:MV_sets}) proves (\ref{eqn:weak_stabilizing}),  and so we have shown that $\beta_k(\C(\cP,r))$ is weakly stabilizing on $\cP$ for all $r \geq 0$.
		
		\vskip0.1truein
		\noindent{\it (ii)  Uniformly  bounded moments:} Via a  calculation similar  to that in Lemma \ref{lem:moments_add_One}, we obtain that, for $m \in [{|A|}/{2},{3|A|}/{2}]$,
		\[ \left|(D_O\beta_k)(\cU_{m,A})\right|
		\ \leq\  2\left[\Bin \left(m,\frac{\omega_dr^d}{|A|}\right)\right]^{k+1}\  \leq \
		2\left[\Bin\left(\left\lceil\frac{3|A|}{2}\right\rceil,\frac{\omega_dr^d}{|A|}\right)\right]^{k+1},\]
		where the inequalities here are to be read as `bounded by a random variable with distribution'.
		Thus, the  uniformly bounded fourth moments for the rightmost binomial random variable implies  uniformly bounded fourth moments for the add-one cost function.
		
		\vskip0.1truein
		\noindent{\it (iii) Polynomial boundedness:} This follows easily from the relation
		\beqq
		\beta_k(\C(\cX,r)) \ \leq\  S_k(\cX,r)\  \leq\  \cX(\mR^d)^{k+1}.
		\eeqq

		From Theorem \ref{thm:clt_Penrose01} and the remarks below it, the above three items suffice to prove the central limit theorem for the Poisson point processes.
		\vskip0.1truein
		\noindent{\it II. Binomial central limit theorem:} \ \
		Given the bounds proven in the previous part of the proof, all that remains to complete the central limit theorem for
		the binomial case is to prove the strong stabilization of $D_O\beta_k$ for $r \notin I_d$.
		
		What we need to show is that there exist a.s.\  finite random variables $\widehat{D}_{\beta_k}(\infty),S$ such that for all finite $\cX \subset B_O(S)^c$,
		\[(D_O\beta_k)((\cP \cap B_O(S)) \cup \cX)\  =\  \widehat{D}_{\beta_k}(\infty). \]
		We shall handle the two case of $r < r_c$ and $r > r_c^*$ separately.

		Assume that $r < r_c$,  or $r \leq r_c$ if $r_c \notin I_d$.  In this case, since $\cCB(\cP,r)$ does not percolate, there are only finitely many components of $\cCB(\cP,r)$ that intersect $B_O(r)$ and all of them are a.s.\ bounded. Let $C_1,\ldots, C_M$ be an enumeration of the components for some a.s.\  finite  $M>0$. (We  exclude the trivial but possible case of $M = 0$.) Further, $C_1,\ldots, C_M$ are a.s.\ bounded subsets and so $C = \bigcup_{i=1}^MC_i$ is also a.s.\ bounded. Thus, in this case we can choose an a.s.\ finite $S$ such that $d(x,C) > 3r$ for all $x \notin B_O(S)$. This implies that for any locally finite $\cX \subset B_O(S)^c$ we have $C \cap \cCB(\cX,r) = \emptyset$. Thus,  for any finite $\cX \subset [B_O(S)]^c$,
		\begin{equation}
		\label{eqn:stabilization_subcritical_Occupied}
		(D_O\beta_k)((\cP \cap B_O(S)) \cup \cX) \ =\  \beta_k(C \cup B_O(r)) - \beta_k(C),
		\end{equation}
		i.e.\ $\beta_k$ strongly stabilizes with stabilization radius $S$ and
		\beqq
		\widehat{D}_{\beta_k}(\infty)\ \definedas \ \beta_k(C \cup B_O(r)) - \beta_k(C).
		\eeqq
		
		Now assume that $r > r_c^*$.  Since $\mR^d \setminus \cCB(\cP,r)$ has only finitely many components that intersect $B_O(r)$,  duality arguments, as in Remark \ref{rem:top_homology},  establish strong stabilization for $\beta_k(\cP,r)$.

		Consequently,  (\ref{eqn:stabilization_subcritical_Occupied}) and  duality establishes strong stabilization of $\beta_k(\cC(\P,r))$ for $r \notin I_d$, and this  completes the proof of the central limit theorem for $\beta_k(\C(\cU_n,r))$.
		
		\vskip0.1truein
		\noindent{\it (III) Positivity of $\tau^2$:} \ \ All that remains is to show the strict positivity of $\tau^2$. By Theorem \ref{thm:clt_Penrose01}, it suffices to show that $D_{\beta_k}(\infty)$ is non-degenerate and this we shall do by using similar arguments to those we used to obtain (\ref{eqn:lower_bound_add_Onea}) and  (\ref{eqn:lower_bound_add_Oneb}).

		Write $\cP_n$ for  $\cP \cap B_O(n)$. We showed  in (\ref{eqn:weak_stabilizing_spheres}) that $|(D_O\beta_k)(\cP_n)| \convas |D_{\beta_k}(\infty)|$, and we have from Lemma \ref{lem:MV_complex_bounds} that, for $n$ large enough,
		\[ |(D_O\beta_k)(\cP_n)|\ \leq \ \sum_{j=k}^{k+1} S_j(\cP_n \cup \{0\},r ; \{0\}) \ \leq\ 2\cP(B_O(2r))^{k+1} .\]
		Since $\EXP{\cP(B_O(2r))^{k+1}} < \infty$, we can use the dominated convergence theorem to obtain that $\EXP{|(D_O\beta_k)(\cP_n)|} \to \EXP{|D_{\beta_k}(\infty)|}$ as $n \to \infty$.
		
		Choose $m$ (depending on $k,d$ only) as in the proof of variance lower bound in Lemma \ref{lem:weak_law_Poisson} (see (\ref{mkballs}) and (\ref{betak-rob}))  and define the set $B^*_m$ as  there. Setting $B^*_{r,m}= rB^*_m$, we have that $|B^*_{m,r}| = |B^*_m|r^{md} > 0$. Thus, for all $n \geq 5r$,
		\begin{eqnarray*}
			&&\EXP{|(D_O\beta_k)(\cP_n)|} \\ &&\qquad \geq \ \EXP{|(D_O\beta_k)(\cP_n)| \1_{\cP_n(B_O(2r))=m, \cP_n(B_O(4r)\setminus B_O(2r)) = 0 }} \no \\
			&&\qquad = \ \EXP{|(D_O\beta_k)(\cP_n)| \, \big| \, \cP_n(B_O(2r))=m, \cP_n(B_O(4r)\setminus B_O(2r)) = 0}  \no \\
			&&\qquad \qquad\times  \pr{\cP_n(B_O(2r))=m, \cP_n(B_O(4r)\setminus B_O(2r)) = 0} \no \\
			&& \qquad\geq\  \pr{\cP_n(B_O(2r))=m, \cP_n(B_O(4r)\setminus B_O(2r)) = 0}  \no \\
			&&\qquad \qquad\times\frac{1}{(\omega_d(2r)^d)^{m}}\int_{(y_1,\ldots,y_{m})\in B_O(2r)^{m}} |(D_O\beta_k)(\{y_1,\ldots,y_{m}\})|\md y_1 \ldots \md y_{m}  \no \\
			&&\qquad \geq \ \frac{|B^*_{m,r}|}{m!}  e^{-\omega_d(4r)^d)}\\
			&&\qquad  >\ 0.
		\end{eqnarray*}
		Thus,
		\[ \EXP{|D_{\beta_k}(\infty)|} \ =\  \lim_{n \to \infty}\EXP{|(D_O\beta_k)(\cP_n)|}\  >\  \frac{|B^*_{m,r}|}{m!}
		e^{-\omega_d(4r)^d)}. \]
		This shows that $\pr{D_{\beta_k}(\infty) \neq 0} > 0$. Thus, to complete the proof of non-degneracy of $D_{\beta_k}(\infty)$, it suffices show that $\pr{D_O\beta_{\infty} = 0} > 0$. %
		\beqq
		\pr{D_{\beta_k}(\infty) = 0} &=& \lim_{n \to \infty}\pr{(D_O\beta_k)(\cP_n) = 0}\\
		& \geq& \lim_{n \to \infty} \pr{\cP_n(B_O(2r))=0}\\
		&=& \exp(-\omega_d(2r)^d)
		\\ & >& 0 .
		\eeqq
		
	\end{proof}

	\noindent

	\section*{Acknowledgements}
	The work has benefitted from discussions with various people. On the topology side, we are thankful to Antonio Rieser and Primoz Skraba. On the probability side, Joseph Yukich answered many questions about the literature on sub-additive and stabilizing functionals, while Mathias Schulte explained to DY the variance lower bound technique of Lemma \ref{lem:lower_bound_variance}. Thanks are also due to his co-authors G\"unter Last and Giovanni Peccatti for sharing their results with us in advance of publication.
	
	\section{Appendices}
	\sec
	\label{sec:App}
	\subsection{Appendix A}
	\label{App:A}
	The following useful lemma needed for variance lower bounds is essentially a simplification of \cite[Theorem 5.2]{Last14} to our situation.
	\begin{lemma}
		\label{lem:lower_bound_variance}
		Let $n \geq 1$ and $\cP_n$ be the Poisson point process with density $nf$, where $f$ satisfies  \eqref{condn:density}. Let $F$ be a translation invariant functional on locally finite point-sets of $\mR^d$ such that $\EXP{F(\cP_n)^2} < \infty$. Assume that there exist  $m \in \mN$, a set $A \subset \text{\rm supp}(f)$, a finite set of points $z_1,\ldots, z_{m}$ and $r > 0$ with $A \oplus B_{z_i}(r) \subset \text{\rm supp}(f)$ for all $i \in \{1,\ldots,m\}$ such that for all $x \in A $ and $(x_1,\ldots,x_m) \in \prod_{i=1}^mB_{x+z_i}(r)$, we have that
		\begin{equation}
		\label{condn:variance_lower_bound}
		\big|\EXP{D_x \big( F(\cP_n \cup \{x_1\ldots,x_m\}) \big)} \big| \ \geq\  c,
		\end{equation}
		for some positive constant $c$. Then
		%
		%
		%
		\beq
		\label{condn:variance_lower_bound-A}
		\VAR{F(\cP_n)}\  \geq f^*n \frac{c^2(f_*(f^*)^{-1})^{m+1}}{8^{m+2}\cdot 4\cdot (m+1)!} \min_{j=1,\ldots,m+1} 2^{-d(m+1-j)}(w_df^*nr^d)^{j-1}|A|
		\eeq
	\end{lemma}
	\begin{proof}
		\cite[Theorem 5.3]{Last14} simplifies \cite[Theorem 5.2]{Last14} to the case of stationary Poisson point processes in Euclidean space. However, since we are dealing with Poisson point processes with non-uniform densities, we require a small change in the arguments there and shall describe this in a moment.  In any case,  the similarity of our lower bound \eqref{condn:variance_lower_bound-A} to that of \cite[Theorem 5.3]{Last14} is to be expected. In particular, if we set $f_* = f^* = \lam$, then \eqref{condn:variance_lower_bound-A} is exactly the variance lower bound in \cite[Theorem 5.3]{Last14} with $t$ there replaced by $\lam n$ here (although the  set $A$ and $r$ are different).
		
		Now, more specifically, set
		\[ U = \{(x,x+z_1+x_1,\ldots,x+z_m+x_m) \: x \in A, \ (x_1,\ldots,x_m) \in \prod_{i=1}^mB_O(r) \}. \]
		Note that this plays the role of $U$ as defined in \cite[Theorem 5.2]{Last14} and defining $g$ on $U$ as
		$$ g(y_1,\ldots,y_{m+1}) := \big|\EXP{F(\cP_n \cup \{y_1,y_2\ldots,y_{m+1}\})} - \EXP{F(\cP_n \cup \{y_2,\ldots,y_{m+1}\})} \big|,$$
		we have that $g > c/2$ Lebesgue a.e.\ on $U$. Setting $\nu_n(.)$ to be the intensity measure of the point process $\cP_n$, we have that
		$$f_*|.| \leq \nu_n(.) \leq f^* |. |,$$
		where we recall that $|.|$ stands for the Lebesgue measure on $\mR^d$. Using these bounds for $\nu_n(.)$, we firstly obtain that
		$$\nu_n^{m+1}(U) \geq nf_*^{m+1}|A|(nw_dr^d)^m.$$
		Secondly, let $\emptyset \neq J \subset \{1,\ldots,m+1\}$ and for $y = (y_1,\ldots,y_{m+1}) \in \mR^{d(m+1)}$, let $y_J$ be the components of $y$ with indices in $J$. If $y \in U$,  then $|y_i - y_j| \leq 2r$ for all $1 \leq i,j \leq m+1$ and so for any $y_J \in \mR^{d|J|}$, we have that
		\[ \nu^{m+1-|J|}_n \big( \{y_{J^c} \in \mR^{d(m+1-|J|)} : y \in U \} \big) \leq (2^df^*nw_dr^d)^{m+1-|J|}. \]

		Using the above upper and lower bounds in the proof of \cite[Theorem 5.3]{Last14} along with the a.e. lower bound for $g$ on $U$, \eqref{condn:variance_lower_bound-A} follows.
	\end{proof}

	\subsection{Appendix B}
	\label{App:B}
	
	We use the notation of Section \ref{sec:clt}, and consider a sequence $\{B_n\}$ of subsets of
	$\mR^d$ satisfying Conditions (A)--(D) there.
	
	Given such a sequence, let $\mathfrak{B}$ ($=\mathfrak{B}(\{B_n\})$) be the collection of all subsets $A$ of $\mR^d$ such that $A = B_n + x$ for some $B_n$ in the  sequence and some point $x \in \mR^d$.
	
	For a $A \in \mathfrak{B}$, we shall denote by $\cU_{m,A}$ the point process obtained by choosing $m$ points uniformly in $A$.  Then the extended binomial process $\cU_n$ of Section \ref{sec:clt} is equivalent to $\cU_{n,B_n}$ in the current notation.

	\begin{theorem}(\cite[Theorem 2.1]{Penrose01})
		\label{thm:clt_Penrose01}
		Let $H$ be a real-valued functional defined for all finite subsets of $\mR^d$ and satisfying the following four conditions:
		\begin{enumerate}
			\item \emph{Translation invariance:} $H(\cX + y) =  H(\cX)$ for all finite subsets $\cX$ and $y \in \mR^d$.
			\item \emph{Strong stabilization:} $H$ is called strongly stabilizing if there exist a.s.\  finite random variables $R$ (called the
			radius of stabilization for $H$) and $D_H(\infty)$ such that, with probability $1$,
			$$(D_OH)\big((\cP \cap B_O(R)) \cup A\big)\  =\  D_H(\infty)$$
			for all finite $A \subset B_O(R)^c$.
			\item \emph{Uniformly bounded moments:}
			\[ \sup_{A \in \mathfrak{B} ; \, 0 \in A} \, \sup_{m \in \left[{|A|}/{2},{3|A|}/{2}\right]}\EXP{\left[(D_OH)(\cU_{m,A})\right]^4}\  <\  \infty  .\]
			\item \emph{Polynomial boundedness:} There is a constant $b_2$ such that, for all finite subsets $\cX \subset \mR^d$,
			\beq
			|H(\cX)| \ \leq\ b_2\left[\text{\rm diam}(\cX) +| \cX|\right]^{b_2}.
			\eeq
		\end{enumerate}
		Then, there exist constants $\sigma^2,\tau^2$ with $0 \leq \tau^2 \leq \sigma^2$, such that, as $n \to \infty$,
		\[ n^{-1}\VAR{H(\cP\cap B_n)} \to \sigma^2, \qquad n^{-1}\VAR{H(\cU_n)} \to \tau^2 ,\]
		\[ n^{-1/2}\left[H(\cP\cap B_n) - \EXP{H(\cP\cap B_n)}\right]\  {\Rar}\  N(0,\sigma^2), \]
		and
		\[ n^{-1/2}\left[H(\cU_n)) - \EXP{H(\cU_n)}\right] \ {\Rar}\  N(0,\tau^2), \]
		where $\Rar$ denotes convergence in distribution. The constants $\sigma^2,\tau^2$ are independent of the choice of $B_n$. If $D_H(\infty)$ is non-degenerate, then $\tau^2 > 0$ and also $\sigma^2 > 0$. Further, if $\EXP{D_H(\infty)} \neq 0$, then $\tau^2 < \sigma^2$.
	\end{theorem}
	The last statement  follows from the remark below \cite[Theorem 2.1]{Penrose01}. In the Poisson case the strongly stabilizing condition required for the central limit theorem can be replaced by the so-called \emph{weak stabilization} condition, as in  \cite[Theorem 3.1]{Penrose01}. $H$ is said to be weakly stabilizing if there exists a random variable $D_{\infty}(H)$ such that
	\beq
	\label{weakstab-rob}
	(D_OH)(\cP \cap A_n)\  \convas\  D_{\infty}(H),
	\eeq
	for any   $\mathfrak{B}$-valued sequence $A_n$ growing to $\mR^d$.
	
	\bibliographystyle{plain}
	\bibliography{limit_theorems_betti_numbers}
	
\end{document}